\documentclass[12pt,reqno]{amsproc}%
\usepackage{amsfonts}
\usepackage{amsmath}
\usepackage{amssymb}
\usepackage{graphicx}%
 \theoremstyle{plain}

\newtheorem{definition}{Definition}[section]

\newtheorem{sublemma}{Sublemma}[subsection]
\newtheorem{lemma}{Lemma}[section]

\newtheorem{proposition}{Proposition}[section]
\newtheorem{remark}{Remark}[section]

\newtheorem{theorem}{Theorem}[section]
 \numberwithin{equation}{subsection}
\begin{document}

\centerline{\Large\bf{Topological Free  Entropy Dimension }}

\vspace{1cm}

\centerline{\Large\bf{  in Unital C$^*$-algebras}}

\vspace{1cm}

\centerline{Don Hadwin \qquad \qquad  and  \qquad \qquad Junhao
Shen\footnote{The second author is supported   by an NSF grant.}}
\bigskip

\centerline{\small{Department of Mathematics and Statistics,
University of New Hampshire, Durham, NH, 03824}}

\vspace{0.2cm}

\centerline{Email: don@math.unh.edu  \qquad \qquad and \qquad \qquad
jog2@cisunix.unh.edu \qquad }

\bigskip

\noindent\textbf{Abstract: } The notion of topological free entropy
dimension of $n$-tuple of elements in a unital C$^*$ algebra was
introduced by Voiculescu. In the paper, we compute   topological
free entropy dimension of one self-adjoint element and topological
free orbit dimension of one self-adjoint element in a unital C$^*$
algebra. We also calculate the values  of topological free entropy
dimensions  of  any families of self-adjoint generators of some
unital C$^*$ algebras, including irrational rotation C$^*$ algebra,
UHF algebra, and minimal tensor product of two reduced C$^*$
algebras of free groups.

\vspace{0.2cm} \noindent{\bf Keywords:} Topological free entropy
dimension, C$^*$ algebra

\vspace{0.2cm} \noindent{\bf 2000 Mathematics Subject
Classification:} Primary 46L10, Secondary 46L54

\section{Introduction}

The theory of  free probability and free entropy   was developed by
Voiculescu from 1990s. It played a crucial role in the recent study
of finite von Neumann algebras (see \cite{BDK}, \cite{Dyk},
\cite{Ge1}, \cite{Ge2}, \cite{GePo}, \cite{GS1}, \cite{GS2},
\cite{HaSh}, \cite{Jung2}, \cite{JuSh}, \cite{V2}, \cite{V3},
\cite{V4}). The analogue of free entropy dimension in C$^*$ algebra
context, the notion of topological free entropy dimension of of
$n-$tuple of elements in a unital C$^*$ algebra, was also introduced
by Voiculescu in
\cite{Voi}. %It was expected the concept of topological free entropy
%dimension will also be useful in the context of C$^*$ algebras.

After introducing the concept of topological free entropy dimension
of $n$-tuple of elements in a unital C$^*$ algebra, Voiculescu
discussed some of its properties  including subadditivity and change
of variables in \cite{Voi}.  In this paper, we will add one basic
property into the list: topological free entropy dimension of one
variable. More specifically, suppose $x$ is a self-adjoint element
in a unital C$^*$ algebra $A$ and $\sigma(x)$ is the spectrum of $x$
in $\mathcal A$. Then topological free entropy dimension of $x$ is
equal to $1-\frac 1 n$ where $n$ is the cardinality of the set
$\sigma(x)$  (see Theorem 4.1).

 In
\cite{Voi}, Voiculescu showed that (i) if $x_1,\ldots,x_n$ is a
family of free semicircular elements in a unital C$^*$ algebra with
a tracial state, then $\delta_{top}(x_1,\ldots, x_n)=n$, where
$\delta_{top}(x_1,\ldots, x_n)$ is the topological free entropy
dimension of $x_1,\ldots,x_n$; (ii) if $x_1,\ldots,x_n$ is the
universal $n$-tuple of self-adjoint contractions, then
$\delta_{top}(x_1,\ldots, x_n)=n$. Except in these two cases, very
few has been known on the values of topological free entropy
dimensions in
other C$^*$ algebras. %One goal of the paper is to compute the values
%of topological free entropy dimensions in other C$^*$ algebras.
%
%It is an open question whether the topological free entropy
%dimension is a C$^*$ algebra invariant, i.e. if $x_1,\ldots,x_n$ and
%$y_1,\ldots, y_m$ are two families of self-adjoint generators of a
%  C$^*$ algebra $\mathcal A$, will the topological free entropy
%  dimension
%of $x_1,\ldots, x_n$ be equal to the one of $y_1,\ldots, y_m$?
Using the inequality between topological free entropy dimension and
Voiculescu's free dimension capacity, we are able to obtain an
estimation  of upper-bound  of topological free entropy dimension
for  a  unital C$^*$ algebra  with  a unique tracial state (see
Theorem 5.1). The lower-bound  of topological free entropy dimension
is also obtained for infinite dimensional simple unital C$^*$
algebra  with a unique tracial state (see Theorem 5.2). As a
corollary, we know that the topological free entropy dimension of
any family of self-adjoint generators of an irrational rotation
C$^*$ algebra or UHF algebra or
$C_{red}^*(F_2)\otimes_{min}C_{red}^*(F_2)$  is equal to 1 (see
Theorem 5.3, 5.4, 5.5). For these C$^*$ algebras, the value of the
topological free entropy dimension is independent of  the choice of
generators.

The rest of the paper is devoted to study another invariant
associated to $n$-tuple of elements in C$^*$ algebras. This
invariant, called topological free orbit dimension, is an analogue
of free orbit dimension in finite von Neumann algebras (see
\cite{HaSh}). We show  that the topological free orbit dimension of
a self-adjoint element in a unital C$^*$ algebra is equal to,
according to some measurement, the packing dimension of the spectrum
of $x$  (see Theorem 7.1).

The organization of the paper is as follows. In the section 2, we
recall  the definition of topological free entropy dimension. Some
technical lemmas are proved in section 3. In section 4, we compute
the topological free entropy dimension of one self-adjoint element
in a unital C$^*$ algebra. In section 5, we study the relationship
between topological free entropy dimension and free capacity
dimension of a unital C$^*$ algebra. Then we show that topological
free entropy dimension of of any family of generators of an infinite
dimensional simple unital C$^*$ algebra with a unique tracial state
is always greater than or equal to 1.
 The concept of topological free orbit dimension of
$n$-tuple of elements in a C$^*$ algebra is introduced in section 6.
Its value for one variable is computed in section 7.

\section{Definitions  and preliminary}

In this section, we are going to recall Voiculescu's definition of
topological free entropy dimension of $n$-tuple of elements in a
unital C$^*$ algebra.
\subsection{A Covering    of  a set  in a metric space}

Suppose $(X,d)$ is a metric space and $K$ is a subset of $X$. A
family of balls in $X$  is called a covering of $K$ if the union of
these balls covers $K$ and the centers of these balls lie in $K$.

\subsection{Covering numbers in  complex matrix algebra
$(\mathcal{M}_{k}(\mathbb{C}))^n$}
 Let $\mathcal{M}_{k}(\mathbb{C})$ be the $k\times k$ full matrix
algebra with entries in $\mathbb{C}$,  and $\tau_{k}$ be the
  normalized trace on $\mathcal{M}_{k}(\mathbb{C})$, i.e.,
  $\tau_{k}=\frac{1}{k}Tr$, where $Tr$ is the usual trace on
  $\mathcal{M}_{k}(\mathbb{C})$.
 Let $\mathcal{U}(k)$ denote the group
of all unitary matrices in $\mathcal{M}_{k}(\mathbb{C})$. Let
$\mathcal{M}_{k}(\mathbb{C})^{n}$ denote the direct sum of $n$
copies of $\mathcal{M}_{k}(\mathbb{C})$.  Let $\mathcal
M_k^{s.a}(\Bbb C)$ be the subalgebra of $\mathcal M_k(\Bbb C)$
consisting of all self-adjoint matrices of $\mathcal M_k(\Bbb C)$.
Let $(\mathcal M_k^{s.a}(\Bbb C))^n$ be the direct sum of $n$ copies
of $\mathcal M_k^{s.a}(\Bbb C)$. Let $\|\cdot \|$ be an  operator
norm  on
  $\mathcal{M}_{k}(\mathbb{C})^{n}$ defined by
  \[
  \Vert(A_{1},\ldots,A_{n})\Vert =  \max\{\|A_1\|,\ldots, \|A_n\| \}    \]
  for all $(A_{1},\ldots,A_{n})$ in $\mathcal{M}_{k}(\mathbb{C})^{n}$. Let $\Vert\cdot\Vert_{2}$
  denote the trace norm induced by $\tau_{k}$ on
  $\mathcal{M}_{k}(\mathbb{C})^{n}$, i.e.,
  \[
  \Vert(A_{1},\ldots,A_{n})\Vert_{2} =\sqrt{\tau_{k}(A_{1}^{\ast}A_{1})+\ldots
  +\tau_{k}(A_{n}^{\ast}A_{n})}
  \]
  for all $(A_{1},\ldots,A_{n})$ in $\mathcal{M}_{k}(\mathbb{C})^{n}$.

For every $\omega>0$, we define the $\omega$-$\|\cdot\|$-ball
$Ball_{}(B_{1},\ldots ,B_{n};\omega,  \|\cdot\|)$ centered at
$(B_{1},\ldots,B_{n})$ in $\mathcal{M}_{k}(\mathbb{C})^{n}$ to be
the subset of $\mathcal{M}_{k}(\mathbb{C})^{n}$   consisting of all
$(A_{1},\ldots,A_{n})$ in $\mathcal{M}_{k}(\mathbb{C})^{n}$ such
that $$\Vert(A_{1},\ldots,A_{n})-(B_{1},\ldots,B_{n})\Vert
<\omega.$$

\begin{definition}
Suppose that $\Sigma$ is a subset of
$\mathcal{M}_{k}(\mathbb{C})^{n}$. We define the covering number
$\nu_\infty(\Sigma, \omega)$ to be the minimal number of
$\omega$-$\|\cdot\|$-balls that consist a covering of $\Sigma$ in
$\mathcal{M}_{k}(\mathbb{C})^{n}$.
\end{definition}

For every $\omega>0$, we define the $\omega$-$\|\cdot\|_2$-ball
$Ball_{}(B_{1},\ldots ,B_{n};\omega,  \|\cdot\|_2)$ centered at
$(B_{1},\ldots,B_{n})$ in $\mathcal{M}_{k}(\mathbb{C})^{n}$ to be
the subset of $\mathcal{M}_{k}(\mathbb{C})^{n}$   consisting of all
$(A_{1},\ldots,A_{n})$ in $\mathcal{M}_{k}(\mathbb{C})^{n}$ such
that $$\Vert(A_{1},\ldots,A_{n})-(B_{1},\ldots,B_{n})\Vert_2
<\omega.$$

\begin{definition}
Suppose that $\Sigma$ is a subset of
$\mathcal{M}_{k}(\mathbb{C})^{n}$. We define the covering number
$\nu_2(\Sigma, \omega)$ to be the minimal number of
$\omega$-$\|\cdot\|_2$-balls that consist a covering of $\Sigma$ in
$\mathcal{M}_{k}(\mathbb{C})^{n}$.
\end{definition}

% For every $R>0$, we define $(\mathcal M_k(\Bbb C)^n)_R$ to be the
%subset of $\mathcal M_k(\Bbb C)^n$ consisting of all these
%$(A_1,\ldots,A_n)$ in $\mathcal M_k(\Bbb C)^n$ such that $\max_{1\le
%j\le n}\| A_j\|\le R$.
\subsection{Noncommutative polynomials}
In this article, we always assume that $\mathcal A$ is a unital
C$^*$-algebra. Let $x_1,\ldots, x_n, y_1,\ldots, y_m$ be
self-adjoint elements in $\mathcal A$. Let $\Bbb C\langle
X_1,\ldots, X_n, Y_1,\ldots,Y_m\rangle $ be the unital
noncommutative polynomials in the indeterminates $X_1,\ldots, X_n,
Y_1,\ldots,Y_m$. Let $\{P_r\}_{r=1}^\infty$ be the collection of all
noncommutative polynomials in $\Bbb C\langle X_1,\ldots, X_n,
Y_1,\ldots,Y_m\rangle $ with rational complex coefficients. (Here
``rational complex coefficients" means that the real   and imaginary
parts of all coefficients of $P_r$ are rational numbers).

\begin{remark}
   We alsways assume that $1\in  \Bbb C\langle
X_1,\ldots, X_n, Y_1,\ldots,Y_m\rangle $.
\end{remark}

\subsection{Voiculescu's Norm-microstates Space}
For all integers $r, k\ge 1$, real numbers $R, \epsilon>0$ and
noncommutative polynomials $P_1,\ldots, P_r$, we define
$$
\Gamma^{(top)}_R(x_1,\ldots, x_n, y_1,\ldots, y_m; k,\epsilon,
P_1,\ldots, P_r)
$$ to be the subset of $(\mathcal M_k^{s.a}(\Bbb C))^{n+m}$
consisting of all these $$ (A_1,\ldots, A_n, B_1,\ldots, B_m)\in
(\mathcal M_k^{s.a}(\Bbb C))^{n+m}
$$ satisfying
 $$  \max \{\|A_1\|, \ldots, \|A_n\|, \|B_1\|,\ldots, \|B_m\|\}\le R
 $$ and
$$
|\|P_j(A_1,\ldots, A_n,B_1,\ldots,B_m)\|-\| P_j(x_1,\ldots,
x_n,y_1,\ldots,y_m)\||\le \epsilon, \qquad \forall \ 1\le j\le r.
$$

\begin{remark}
In the definition of norm-microstates space, we use the following
assumption. If
$$
P_j(x_1,\ldots, x_n,y_1,\ldots,y_m)= \alpha_0\cdot I_{\mathcal A}+
\sum_{s=1}^N\sum_{1\le i_1,\ldots, i_s\le n+m} \alpha_{i_1\cdots
i_s}z_{i_1}\cdots z_{i_s}
$$ where $z_1,\ldots, z_{n+m}$ denotes $x_1,\ldots, x_n, y_1,\ldots,
y_m$ and $\alpha_0,\alpha_{i_1\cdots i_s}$ are in $\Bbb C$, then
$$
P_j(A_1,\ldots, A_n,B_1,\ldots,B_m)= \alpha_0\cdot I_k+
\sum_{s=1}^N\sum_{1\le i_1,\ldots, i_s\le n+m} \alpha_{i_1\cdots
i_s}Z_{i_1}\cdots Z_{i_s}
$$ where $Z_1,\ldots, Z_{n+m}$ denotes $A_1,\ldots, A_n, B_1,\ldots,
B_m$ and $I_k$ is the identity matrix in $\mathcal M_k(\Bbb C)$.
\end{remark}

\begin{remark} In the original definition of norm-microstates space
in \cite{Voi}, the parameter $R$ was not introduced. Note the
following  observation:   Let $R> \max\{\|x_1\|,\ldots, \|x_n\|,\|
y_1\|,\ldots, \|y_m\|\}$. When $r$ is large enough so that $$
\{X_1,\ldots,X_n,Y_1,\ldots,Y_m\}\subset \{P_1,\ldots,P_r\}
$$ and $0<\epsilon<R- \max\{\|x_1\|,\ldots, \|x_n\|,\|
y_1\|,\ldots, \|y_m\|\}$, we have
$$
\Gamma^{(top)}_{R}(x_1,\ldots, x_n, y_1,\ldots, y_m; k,\epsilon,
P_1,\ldots, P_r)=\Gamma_{top} (x_1,\ldots, x_n, y_1,\ldots, y_m;
k,\epsilon, P_1,\ldots, P_r)
$$ for all $k\ge 1$, where $\Gamma_{(top)} (x_1,\ldots, x_n, y_1,\ldots, y_m;
k,\epsilon, P_1,\ldots, P_r)$ is the norm-microstates space defined
in \cite{Voi}. Thus our definition agrees with the one in \cite{Voi}
for large $R$, $r$ and small $\epsilon$.

In the later sections, we need to construct the ultraproduct of some
matrix algebras, it will be  convenient for us to include the
parameter ``R" in the definition of norm-microstate space.
\end{remark}

Define the norm-microstates space of $x_1,\ldots,x_n$ in the
presence of $y_1,\ldots, y_m$, denoted by
$$
\Gamma^{(top)}_R(x_1,\ldots, x_n: y_1,\ldots, y_m; k,\epsilon,
P_1,\ldots, P_r)
$$ as the projection of $
\Gamma^{(top)}_R(x_1,\ldots, x_n, y_1,\ldots, y_m; k,\epsilon,
P_1,\ldots, P_r) $ onto the space $(\mathcal M_k^{s.a}(\Bbb C))^n$
via the mapping
$$
(A_1,\ldots, A_n, B_1,\ldots, B_m) \rightarrow (A_1,\ldots, A_n).
$$

\subsection{Voiculescu's topological free entropy dimension (see \cite{Voi})}

Define
$$
\nu_\infty(\Gamma^{(top)}_R(x_1,\ldots, x_n: y_1,\ldots, y_m;
k,\epsilon, P_1,\ldots, P_r),\omega)
$$ to be the covering number of the set $
\Gamma^{(top)}_R(x_1,\ldots, x_n: y_1,\ldots, y_m; k,\epsilon,
P_1,\ldots, P_r) $ by $\omega$-$\| \cdot\|$-balls in the metric
space $(\mathcal M_k^{s.a}(\Bbb C))^n$ equipped with operator norm.

\begin{definition}Define
$$
  \begin{aligned}
     \delta_{top}(x_1,\ldots, & \ x_n: y_1,\ldots, y_m;  \omega)\\
     & =
     \sup_{R>0} \ \inf_{\epsilon>0, r\in \Bbb N}  \ \limsup_{k\rightarrow\infty} \frac {\log(\nu_\infty(\Gamma^{(top)}_R(x_1,
     \ldots, x_n: y_1,\ldots, y_m; k,\epsilon, P_1,\ldots,
P_r),\omega))}{-k^2\log\omega}. \end{aligned}
$$

\noindent {\bf The topological entropy dimension } of $x_1,\ldots,
x_n$ in the presence of $y_1,\ldots, y_m$ is defined by
$$ \delta_{top}(x_1,\ldots, \ x_n: y_1,\ldots, y_m)=
 \limsup_{\omega\rightarrow 0^+}  \delta_{top}(x_1,\ldots,   x_n:y_1,\ldots, y_m;
 \omega).
$$
\end{definition}
\begin{remark} Let $M>
\max\{\|x_1\|,\ldots,\|x_n\|,\|y_1\|,\ldots,\|y_m\|\}$ be some
positive number. By Remark 2.3, we know
$$
  \begin{aligned}
     \delta_{top}(x_1,\ldots, & \ x_n: y_1,\ldots, y_m )\\
     & =
       \limsup_{\omega\rightarrow 0^+}  \inf_{\epsilon>0, r\in \Bbb N}  \ \limsup_{k\rightarrow\infty} \frac {\log(\nu_\infty(\Gamma^{(top)}_M(x_1,
     \ldots, x_n: y_1,\ldots, y_m; k,\epsilon, P_1,\ldots,
P_r),\omega))}{-k^2\log\omega}. \end{aligned}
$$
\end{remark}

\subsection{C$^*$ algebra ultraproduct  and  von Neumann
algebra ultraproduct}

Suppose $\{\mathcal M_{k_m}(\Bbb C)\}_{m=1}^\infty$ is a sequence of
complex matrix algebras where $k_m$ goes to infinity when $m$
approaches infinity. Let $\gamma$ be a free ultrafilter in
$\beta(\Bbb N)\setminus \Bbb N$. We can introduce a unital C$^*$
algebra $\prod_{m=1}^\infty \mathcal M_{k_m}(\Bbb C)$ as follows:
$$
\prod_{m=1}^\infty \mathcal M_{k_m}(\Bbb C) = \{(Y_m)_{m=1}^\infty \
| \ \forall \ m\ge 1, \ Y_m \in \mathcal M_{k_m}(\Bbb C) \ \text {
and } \ \sup_{m\ge 1} \| Y_m\|<\infty\}.
$$
We can also introduce the  norm   closed two sided ideals $\mathcal
I_\infty$ and $\mathcal I_2$   as follows.
$$
  \begin{aligned}
    \mathcal I_\infty & = \{(Y_m)_{m=1}^\infty\in  \prod_{m=1}^\infty \mathcal M_{k_m}(\Bbb C) \
| \ \lim_{m\rightarrow \gamma } \|Y_m\| =0\}\\
 \mathcal I_2 & = \{(Y_m)_{m=1}^\infty\in  \prod_{m=1}^\infty \mathcal M_{k_m}(\Bbb C) \
| \ \lim_{m\rightarrow \gamma } \|Y_m\|_2 =0\}
  \end{aligned}
$$

\begin{definition}
The C$^*$ algebra ultraproduct of $\{\mathcal M_{k_m}(\Bbb
C)\}_{m=1}^\infty$ along the ultrfilter $\gamma$, denoted by
$\prod_{m=1}^\gamma \mathcal M_{k_m}(\Bbb C)$, is defined to be the
quotient
  algebra of $\prod_{m=1}^\infty \mathcal M_{k_m}(\Bbb C)$ by
the ideal $\mathcal I_\infty$. The image of $(Y_m)_{m=1}^\infty\in
\prod_{m=1}^\infty \mathcal M_{k_m}(\Bbb C)$ in the quotient algebra
is denoted by $[(Y_m)_{m}]$.
\end{definition}

\begin{definition}
The von Neumann algebra ultraproduct of $\{\mathcal M_{k_m}(\Bbb
C)\}_{m=1}^\infty$ along the ultrfilter $\gamma$, also denoted by
$\prod_{m=1}^\gamma \mathcal M_{k_m}(\Bbb C)$ if   no confusion
arises, is defined to be the quotient    algebra of
$\prod_{m=1}^\infty \mathcal M_{k_m}(\Bbb C)$ by the ideal $\mathcal
I_2$. The image of $(Y_m)_{m=1}^\infty\in \prod_{m=1}^\infty
\mathcal M_{k_m}(\Bbb C)$ in the quotient algebra is denoted by
$[(Y_m)_{m}]$.
\end{definition}
\begin{remark}
The von Neumann algebra ultraproduct   $\prod_{m=1}^\gamma \mathcal
M_{k_m}(\Bbb C)$   is a finite  factor (see \cite{McDuff}).
\end{remark}

\subsection{Topological free
entropy dimension of elements in a non-unital C$^*$ algebra}

Topological free entropy dimension can also be defined for $n$-tuple
of elements in a non-unital C$^*$ algebra. Suppose that $\mathcal A$
is a non-unital C$^*$-algebra. Let $x_1,\ldots, x_n, y_1,\ldots,
y_m$ be self-adjoint elements in $\mathcal A$. Let $\Bbb C\langle
X_1,\ldots, X_n, Y_1,\ldots,Y_m\rangle \ominus \Bbb C$ be the
noncommutative polynomials in the indeterminates $X_1,\ldots, X_n,
Y_1,\ldots,Y_m$ without constant terms. Let $\{P_r\}_{r=1}^\infty$
be the collection of all noncommutative polynomials in $\Bbb
C\langle X_1,\ldots, X_n, Y_1,\ldots,Y_m\rangle\ominus \Bbb C $ with
rational complex coefficients. Then norm-mocrostate space
$$
\Gamma^{(top)}_R(x_1, \ldots, x_n: y_1,\ldots, y_m; k,\epsilon,
P_1,\ldots, P_r)
$$ can be defined similarly as in section 2.4.
So topological free entropy dimension
$$\delta_{top}(x_1,\ldots,x_n:y_1,\ldots,y_m)$$ can also be defined
similarly as in section 2.5.

In the paper, we will focus on the case when $\mathcal A$ is a
unital C$^*$ algebra.
\section{Some technical  lemmas}

\subsection{}Suppose $x$ is a self-adjoint element in a unital C$^*$
algebra $\mathcal A$. Let $\sigma(x)$ be the spectrum of $x$ in
$\mathcal A$.

\begin{theorem}Let $R> \|x\|$.
For any $\omega>0$, we have the following.
\begin{enumerate}
  \item   There are some integer $n\ge 1$ and distinct real numbers  $\lambda_1,\lambda_2, \cdots, \lambda_n$
  in $\sigma(x)$ satisfying (i) $|\lambda_i-\lambda_j|\ge \omega$ for all $1\le i\ne j\le n$; and
 (ii)  for any $\lambda$ in $\sigma(x)$, there is some $\lambda_j$
 with $1\le j\le n$
 such that $|\lambda-\lambda_j|\le \omega$.
  \item    There are some   $r_0>0$ and $\epsilon_0>0$ such that the following holds: when
   $r>r_0$, $\epsilon<\epsilon_0$, for any $A$ in
   $
\Gamma^{(top)}_R(x; k,\epsilon, P_1,\ldots, P_r) $, there are
positive integers $1\le k_1,\ldots, k_n\le k$ with
$k_1+k_2+\cdots+k_n=k$ and some unitary matrix $U$ in $\mathcal
M_k(\Bbb C)$ satisfying
$$
\|U^*AU- \left (  \begin{aligned}
   \lambda_1 I_{k_1} \quad  & \quad 0  \quad & \quad  \cdots  \quad &  \qquad  0 \\
    0 \quad  & \quad  \lambda_2 I_{k_2} \quad  & \quad  \cdots  \quad  & \qquad  0 \\
    \cdots \quad &  \quad \cdots  \quad  &  \quad \ddots  \quad  &
    \quad \cdots \\
    0  \quad &  \quad 0 \quad  &  \quad \cdots  \quad &
    \quad  \lambda_{n } I_{k_{n}}
\end{aligned} \right )\| \le 2\omega,
$$ where  $I_{k_j}$ is the $k_j\times k_j$ identity matrix in $\mathcal M_{k_j}(\Bbb C)$  for $1\le
j\le n$.
\end{enumerate}

\end{theorem}

\begin{proof}
The proof of  part (1) is trivial. We will only prove part (2).
Assume that the result in (2) does not hold. Then there is some
$\omega>0$ so that the following holds: for all $m\ge 1$, there are
$k_m\ge 1$ and some $A_m$ in $ \Gamma^{(top)}_R(x; k_m,\frac 1 {m},
P_1,\ldots, P_{m}) $ such that
\begin{equation}
\|U^*A_mU- \left (  \begin{aligned}
   \lambda_1 I_{s_1} \quad  & \quad 0  \quad & \quad  \cdots  \quad  & \qquad    0 \\
    0 \quad  & \quad  \lambda_2 I_{s_2} \quad  & \quad  \cdots  \quad  & \qquad  0 \\
   \cdots \quad &  \quad \cdots  \quad  &  \quad \ddots  \quad  &
  \quad  \cdots \\
    0  \quad &  \quad 0 \quad  &  \quad \cdots  \quad  &
    \quad \lambda_{n} I_{s_{n}}
\end{aligned} \right )\| > 2\omega, \tag{$*$}
\end{equation} for   every  $1\le s_1,\ldots,s_n\le k_m$
with $s_1+\cdots +s_n=k_m$ and every unitary matrix $U$ in $\mathcal
M_{k_m}(\Bbb C)$.

Let $\gamma$ be a free ultrafilter in $\beta(\Bbb N)\setminus \Bbb
N$. Let $\mathcal B=\prod_{m=1}^\gamma \mathcal M_{k_m}(\Bbb C)$ be
the C$^*$ algebra ultraproduct  of $\{\mathcal M_{k_m}(\Bbb
C)\}_{m=1}^\infty$ along the ultrafilter $\alpha$, i.e.
$\prod_{m=1}^\gamma \mathcal M_{k_m}(\Bbb C)$ is the quotient
algebra of the C$^*$ algebra $\prod_{m=1}^\infty \mathcal
M_{k_m}(\Bbb C)$ by $\mathcal I_\infty$, the $0$-ideal of the norm
$\|\cdot \|$, where $\mathcal I_\infty=\{(A_m)_{m=1}^\infty \in
\prod_{m=1}^\infty \mathcal M_{k_m}(\Bbb C) \ | \
\lim_{m\rightarrow \gamma}\|A_m \|=0\}$. Let
$a=[(U^*A_mU)_{m=1}^\infty]$ be a self-adjoint element
  in $\mathcal B$. By mapping $x$ to $a$, there is a
 unital $*$-isomorphism from the C$^*$ subalgebra generated by $\{I_{\mathcal A},x\}$ in
  $\mathcal A$ onto the C$^*$ subalgebra generated by $\{I_{\mathcal B},a\}$ in $\mathcal B$. Thus $ \sigma(x)= \sigma(a)$.
It is
  not hard to see that Hausdorff-dist$(\sigma(U^*A_mU), \sigma(a))\rightarrow
  0$ as $m$ goes to $\gamma$,
   which contradicts with the results in
  part (1) and ($*$).

\end{proof}

\subsection{}%{Covering numbers of sets in $\mathcal M_k(\Bbb C)$}
In this subsection, we will use the following notation.
\begin{enumerate}
\item  [(i)] Let $n,m$ be some positive integers with
$ n\ge m$.

\item  [(ii)] Let  $\delta$, $\theta$ be some positive numbers.

 \item  [(iii)] Let  $\{\lambda_1,
\lambda_2,\ldots, \lambda_m\}\cup
\{\lambda_{m+1},\ldots,\lambda_n\}$ be a family of real numbers such
that
$$|\lambda_i-\lambda_j|\ge \theta\qquad \text { for all $1\le i<j\le
m$.}$$

\item  [(iv)] Let   $k$ be a  positive integer such that $k-(n-m)$ is divided by
$m$. We let $$t= \frac {k-n+m}{m}.$$

\item  [(v)] We let
$$B=diag(\lambda_{m+1},\ldots, \lambda_n)$$ be a diagonal matrix in
$\mathcal M_{n-m}(\Bbb C)$ and $$A= diag(\lambda_1 I_t, \lambda_2
I_t, \ldots, \lambda_m I_t, B)$$ be a block-diagonal matrix in
$\mathcal M_k(\Bbb C),$ where $I_t$ is the identity matrix in
$\mathcal M_t(\Bbb C)$.

\item  [(vi)] We let  $A$ be  defined as above and
$$
\Omega(A)= \{U^*AU\ | \ U \text { is in } \mathcal U(k)\}.
$$

\item  [(vii)] Assume that $\{e_{ij}\}_{i,j=1}^k$ is a canonical basis of
$\mathcal M_k(\Bbb C)$. We let $$
  \begin{aligned}
  V_1 &= span \{ e_{ij}\ | \ |\lambda_{[\frac i m]+1}-\lambda_{[\frac j
     m]+1}| \ge \theta, \ \text { with } \ 1\le i,j< mt\}; \quad \text { and }\\
     V_2&=
     \mathcal M_k(\Bbb C) \ominus  V_1 ,
  \end{aligned}
$$ where $[\frac i m]$, or $[\frac j m]$, denotes the largest
integer $\le$ $[\frac i m]$, or  $[\frac j m]$ respectively.
\end{enumerate}
\begin{lemma} We follow the notation as above. Suppose \  $\|U_1AU_1^*- U_2AU_2^*\|_2\le   \delta  $  for some
unitary matrices $U_1$ and $U_2$ in $\mathcal U(k)$. Then the
following hold.
\begin{enumerate}
\item There exists  some $S \in V_2$ such that $\|S\|_2\le 1$
and
$$\|U_1- U_2S\|_2\le \frac {  \delta}{ \theta}. $$
\item If $n=m$, then there is a unitary matrix $W$  in $V_2$ such that
$$\|U_1- U_2W \|_2\le \frac {  3\delta}{ \theta} .
$$\end{enumerate}
\end{lemma}
\begin{proof} Assume that
$$U_2^*U_1 =\left (
   \begin{aligned}
     U_{11} \quad & \quad U_{12} \quad & \quad \cdots \quad &\quad
     U_{1, m+1}\\
     U_{21} \quad & \quad U_{22} \quad & \quad \cdots \quad &\quad
     U_{2, m+1}\\
     \cdots \quad & \quad \cdots \quad & \quad \cdots \quad &\quad
     \cdots\\
     U_{m+1,1} \quad & \quad U_{m+1,2} \quad & \quad \cdots \quad &\quad
     U_{m+1, m+1}\\
   \end{aligned}
 \right )$$ where
$U_{i,j}$ is a $t\times t$ matrix, $U_{i,m+1}$ a $t\times (n-m)$
matrix,  $U_{ m+1, j}$ a $(n-m)\times t$ matrix  for $1\le i,j\le m$
and $U_{m+1,m+1}$ is a $(n-m)\times (n-m)$ matrix.

(1) \quad Let$$ S =\left ( \begin{aligned}
     U_{11} \quad & \quad 0 \quad & \quad \cdots \quad &\quad 0 \quad  &\quad
     U_{1, m+1}\\
     0 \quad & \quad U_{22} \quad & \quad \cdots \quad  &\quad 0 \quad &\quad
     U_{2, m+1}\\
     \cdots \quad & \quad \cdots \quad & \quad \ddots \quad & \quad \cdots \quad&\quad
     \cdots\\
    0 \quad & \quad 0 \quad & \quad \cdots \quad &\quad U_{m ,m} \quad &\quad
     U_{m , m+1} \\
     U_{m+1,1} \quad & \quad U_{m+1,2} \quad & \quad \cdots \quad &\quad U_{m+1,m} \quad &\quad
     U_{m+1, m+1} \\
   \end{aligned}\right).
$$ It is easy to see that $S$ is in $V_2$, $\| S\|_2\le 1$ and
$$
   \begin{aligned}
    \delta^2 &\ge    \|U_1AU_1^*- U_2AU_2^*\|_2^2 = \frac 1 k Tr((U_2^*U_1A-
    AU_2^*U_1)(U_2^*U_1A-
    AU_2^*U_1)^*)\\
    &\ge \frac 1 k \sum_{1\le i\ne j\le m}Tr(|\lambda_i-\lambda_j|^2U_{ij}U_{ij}^*)\\
    &\ge \frac 1 k \cdot \theta^2 \sum_{1\le i\ne j\le m} Tr(U_{ij}U_{ij}^*).
   \end{aligned}
$$
 Hence
$$\| U_1- U_2S \|_2^2=\| U_2^*U_1-S \|_2^2=\frac 1 k \sum_{1\le i\ne j\le m}  Tr(U_{ij}U_{ij}^*)\le \frac {  \delta^2}{\theta^2}.
$$
It follows that $$\| U_1- U_2S \|_2\le \frac {  \delta}{\theta}.$$

 (2) \quad If $n=m$, then
   $$V_2= \mathcal M_t(\Bbb C) \oplus \mathcal M_t(\Bbb C) \oplus
   \cdots \oplus \mathcal M_t(\Bbb C).$$
By the construction of $S$, we can assume   $S=WH$ is a polar
decomposition of $S$ in $V_2$ for some unitary matrix $W$ and
positive matrix $H$ in $V_2$. Again by the construction of $S$, we
know that $\|S\|\le 1$, whence $\|H\|\le 1.$ From the proven fact
that $\|U_2^*U_1-S\|_2\le \frac{\delta}{\theta},$ we know that
$$ \|H^2-I\|_2= \|S^*S-I\|_2\le \frac {2\delta}\theta.$$  Thus
$$
\|H-I\|_2\le \|H^2-I\|_2\le \frac {2 \delta}{\theta}.
$$ It follows that
$$
\|U_1-U_2W\|_2\le \|U_1-U_2WH\|_2+\|U_2WH-U_2W\|_2=\|U_1-U_2S\|_2+\|
H-I\|_2\le \frac {3\delta}{\theta}.
$$
\end{proof}

\begin{lemma} We have the following results.
\begin{enumerate}
\item   For every $U\in \mathcal U(k)$, let
$$
\Sigma(U)= \{W\in \mathcal U(k)  \ | \ \exists \ S  \in V_2 \text{
such that } \ \|S\|_2\le 1 \ \text { and } \   \ \|W- US\|_2\le
\frac { \delta}{\theta}\}.
$$  Then
 the volume of
$\Sigma(U)$ is bounded above by
$$
\mu(\Sigma(U)) \le (C_1\cdot 4\delta/\theta)^{k^2}\cdot  \left
(\frac{C\theta}{ \delta}\right )^{2mt^2+4m(n-m)t+2(n-m)^2},
$$ where $\mu$ is the normalized Haar measure on the unitary group $\mathcal
U(k)$  and $C, C_1$ are some   constants independent of $ k,
\delta,\theta$.
\item When $n=m$,  for every $U\in \mathcal U(k)$, let
$$\begin{aligned}
\tilde\Sigma(U)&= \{W\in \mathcal U(k)  \ | \ \exists \text{ a
unitary matrix $W_1$   in $V_2$}  \text{ such that $\|W-UW_1 \|_2\le
\frac { 3 \delta}{ \theta}$ }\}.\end{aligned}
$$ Then
$$
\mu(\tilde\Sigma(U)) \le (C_1\cdot 8\delta/\theta)^{k^2}\cdot  \left
(\frac{C\theta}{ \delta}\right )^{mt^2 },
$$
\end{enumerate}
\end{lemma}
\begin{proof}(1) \ By computing the covering number of the set $\{S  \ | \ S
\in V_2,   \text{ such that } \|S\|_2\le 1 \}$ by $
\delta/\theta$-$\|\cdot\|_2$-balls in $\mathcal M_k(\Bbb C)$, we
know
$$\begin{aligned}
\nu_2(\{S \ | \ S\in V_2, \|S\|_2\le 1\},\frac \delta \theta)& \le
\left (\frac{C\theta}{\delta}\right )^{\text {real dimension of
  of $V_2$}}\\ &\le \left
(\frac{C\theta}{ \delta}\right )^{2mt^2+4m(n-m)t+2(n-m)^2},
\end{aligned}
$$ where $C$ is a universal constant. Thus the covering number of the set $\Sigma(U)$ by the $4\delta/\theta$-$\|\cdot\|_2$-balls
 in $\mathcal M_k(\Bbb C)$  is bounded by
$$
\nu_2(\Sigma(U), \frac {4\delta}{\theta}) \le \nu_2(\{S \ | \ S\in
V_2, \|S\|_2\le 1\},\frac \delta \theta)\le\left (\frac{C\theta}{
\delta}\right )^{2mt^2+4m(n-m)t+2(n-m)^2}.
$$
But the ball of radius $4\delta/\theta$ in $\mathcal U(k)$ has the
volume bounded by
$$
\mu(\text {ball of radius $4\delta/\theta$})\le (C_1\cdot
4\delta/\theta)^{k^2},
$$ where $C_1$ is a universal constant.
Thus
$$
\mu(\Sigma(U)) \le (C_1\cdot 4\delta/\theta)^{k^2}\cdot  \left
(\frac{C\theta}{ \delta}\right )^{2mt^2+4m(n-m)t+2(n-m)^2}.
$$

(2) A slight adaption of the proof of part (1) gives us the proof of
part (2). \end{proof}
\begin{lemma} Let $\Omega(A)$ be
defined as in   (vi) at the beginning of this subsection.
\begin{enumerate}
\item The covering number of $\Omega(A)$ by the $  \frac 1 2\delta $-$\|\cdot\|_2$-balls in $\mathcal M_k(\Bbb C)$
is bounded  below by
$$
\nu_2(\Omega(A), \frac 1 2\delta) \ge (C_1\cdot
4\delta/\theta)^{-k^2}\cdot \left (\frac{C\theta}{ \delta}\right
)^{-(2mt^2+4m(n-m)t+2(n-m)^2)}
$$
\item If $n=m$, then
$$
\nu_2(\Omega(A),\frac 1 2 \delta ) \ge (C_1\cdot
8\delta/\theta)^{-k^2}\cdot \left (\frac{C\theta}{ \delta}\right
)^{-mt^2 }.
$$
\end{enumerate}
\end{lemma}

\begin{proof} (1) \
For every $U\in \mathcal U(k)$, define
$$
\Sigma(U)= \{W\in \mathcal U(k)  \ | \ \exists \ S=S^*\in V_1,
\text{ such that } \|S\|_2\le 1,  \ \|W- US\|_2\le {\frac {
\delta}{\theta}}\}.
$$
By preceding lemma, we have
$$
\mu(\Sigma(U)) \le (C_1\cdot 4\delta/\theta)^{k^2}\cdot  \left
(\frac{C\theta}{ \delta}\right )^{mt^2+4m(n-m)t+2(n-m)^2}.
$$
 A ``parking" (or  exhausting) argument will show
the existence of a family of unitary elements $\{ U_i
\}_{i=1}^N\subset \mathcal U(k)$ such that
$$
N\ge (C_1\cdot 4\delta/\theta)^{-k^2}\cdot  \left (\frac{C\theta}{
\delta}\right )^{-(mt^2+4m(n-m)t+2(n-m)^2)}
$$  and
$$
  U_i \  \text { is not contained in } \cup_{j=1}^{i-1}\Sigma (U_j), \qquad \forall \ i=1,\ldots, N.
$$

From the definition of each $\Sigma(U_j)$, it follows that  $$ \|
U_i-U_jS\|_2> {\frac { \delta}{\theta}} ,\qquad \forall \ S\in V_2,
\text{ with } \|S\|_2\le 1, \ \forall 1\le j<i \le N.
$$
By Lemma 3.1, we know that
$$
\|U_iAU_i^*-U_jAU_j^*\|_2>  \delta  ,\qquad  \forall 1\le j<i\le N,
$$
which implies that
$$
\nu_2(\Omega(A),\frac 1 2\delta) \ge N\ge(C_1\cdot
4\delta/\theta)^{-k^2}\cdot \left (\frac{C\theta}{ \delta}\right
)^{-(mt^2+4m(n-m)t+2(n-m)^2)}
$$

(2) is similar as (1).
\end{proof}

\subsection{} We have following theorem.
\begin{theorem}Let $n\ge m$, $\delta,\theta>0$ and $\{\lambda_1,
\lambda_2,\ldots, \lambda_m\}\cup
\{\lambda_{m+1},\ldots,\lambda_n\}$ be a family of real numbers such
that
$$|\lambda_i-\lambda_j|\ge \theta$$ for all $1\le i<j\le m$. Let
$k$ be a  positive integer such that $k-(n-m)$ is divided by $m$ and
  $$t= \frac {k-n+m}{m}.$$ Let
$$B=diag(\lambda_{m+1},\ldots, \lambda_n)$$ be a diagonal matrix in
$\mathcal M_{n-m}(\Bbb C)$ and $$A= diag(\lambda_1 I_t, \lambda_2
I_t, \ldots, \lambda_m I_t, B)$$ be a block-diagonal matrix in
$\mathcal M_k(\Bbb C),$ where $I_t$ is the identity matrix in
$\mathcal M_t(\Bbb C)$.
 We let
$$
\Omega(A)= \{U^*AU\ | \ U \text { is in } \mathcal U(k)\}.
$$

Then the covering number of $\Omega(A)$ by the $ \frac 1 2\delta
$-$\|\cdot\|$-balls in $\mathcal M_k(\Bbb C)$  is bounded  below by
$$
\nu_\infty(\Omega(A), \frac 1 2\delta  ) \ge (C_1\cdot
4\delta/\theta)^{-k^2}\cdot \left (\frac{C\theta}{ \delta}\right
)^{-(2mt^2+4m(n-m)t+2(n-m)^2)},
$$ where $C, C_1$ are some universal constants.

When $n=m$, we have
$$ \nu_\infty(\Omega(A), \frac 1 2\delta ) \ge (C_1\cdot
8\delta/\theta)^{-k^2}\cdot \left (\frac{C\theta}{ \delta}\right )^{
-mt^2 },
$$
\end{theorem}
\begin{proof}
Note that
$$
\nu_\infty(\Omega(A),\frac{ \delta}{2} ) \ge \nu_2(\Omega(A),\frac{
\delta}{2}  ), \qquad \forall \ \delta>0.$$ The result follows
directly from preceding lemma.
\end{proof}

\subsection{}The following proposition, whose proof is skipped, is
an easy extension of Lemma 3.3.

\begin{proposition}
    Let $m, k$ be some positive integers and $\theta, \delta$ be some positive numbers.
Let $T_1,T_2,\ldots, T_{m+1}$ is a partition of the set
$\{1,2,\ldots,k\}$, i.e. $\cup_{i=1}^{m+1} T_i= \{1,2,\ldots,k\}$
and $T_i\cap T_j=\emptyset$ for $1\le i\ne j\le m+1$.  Let
$\lambda_1,\ldots, \lambda_k$ be some real numbers
    such that, if $1\le j_1\ne j_2\le m$ then
    $$
       |\lambda_{i_1}-\lambda_{i_2}|>\theta, \qquad \forall \ i_1\in
      T_{j_1}, \ i_2\in T_{j_2}.
    $$
    Let $A=diag(\lambda_1,\lambda_2,\ldots, \lambda_k)$ be a
    self-adjoint matrix in $\mathcal M_k(\Bbb C)$ and
    $$
     \Omega(A)= \{ U^* A U\ | \ U\in \mathcal U(k)\}
    $$ be a subset of $\mathcal M_k(\Bbb C)$.

 Let $s_j$ be the cardinality of the set $T_j$ for $1\le j\le m+1$.
Then the covering number of $\Omega(A)$ by the $ \frac 1 2\delta
$-$\|\cdot\|_2$-balls in $\mathcal M_k(\Bbb C)$  is bounded  below
by
$$
\nu_2(\Omega(A), \frac 1 2\delta  ) \ge (C_1\cdot
4\delta/\theta)^{-k^2}\cdot \left (\frac{C\theta}{ \delta}\right
)^{-2 s_1^2-2s_2^2-\cdots-2s_{m+1}^2-4 (s_1+\cdots+s_m)s_{m+1}},
$$ where $C, C_1$ are some universal constants.

\end{proposition}
\section{Topological free entropy dimension of one variable}
Suppose $x$ is a self-adjoint element of a unital C$^*$ algebra
$\mathcal A$. In this section, we are going to compute the
topological entropy dimension of $x$.

\subsection{Upperbound}

\begin{proposition}
Suppose $x$ in $\mathcal A$ is a self-adjoint element with the
spectrum $\sigma(x)$. Then
$$
\delta_{top}(x) \le 1- \frac 1 n,
$$ where $n$ is the cardinality of $ \sigma(x)$. Here we
assume that $\frac 1 \infty =0.$
\end{proposition}
\begin{proof}By \cite{Voi}, we know that the inequality always holds when $n$ is infinity. We
    need only to show that $$
\delta_{top}(x) \le 1- \frac 1 n, $$ when $n<\infty$.

Assume that $\lambda_1,\ldots, \lambda_{n}$ are in the spectrum of
$x$ in $\mathcal A$.

 Let $R>\|x\|$. By
Theorem 3.1, for every $\omega>0$, there are $r_0>0$ and
$\epsilon_0>0$ such that, for all $r>r_0, \epsilon<\epsilon_0$,
$$A\ \in  \Gamma^{(top)}_R(x; k,\epsilon, P_1,\ldots, P_r), $$ there are
some $1\le k_1,\ldots, k_{n }\le k$,   with $k_1+\cdots  + k_{n }=k$
and   a unitary matrix $U$  in $\mathcal M_k(\Bbb C)$  satisfying
\begin{equation}
\left \|A- U\left (  \begin{aligned}
   \lambda_1 I_{k_1} \quad  & \quad 0  \quad & \quad  \cdots  \quad & \quad  0 \\
    0 \quad  & \quad  \lambda_2 I_{k_2} \quad  & \quad  \cdots  \quad &  \quad 0 \\
    0  \quad &  \quad 0 \quad  &  \quad \cdots  \quad &  \quad \lambda_n
    I_{k_{n}}
\end{aligned} \right )U^* \right \| \le 2 \omega.\tag{$**$}
\end{equation}

Let $$ \Omega(k_1,\ldots,k_n) = \left \{ U\left (
\begin{aligned}
   \lambda_1 I_{k_1} \quad  & \quad 0  \quad & \quad  \cdots  \quad & \quad  0&    0\qquad\\
    0 \quad  & \quad  \lambda_2 I_{k_2} \quad  & \quad  \cdots  \quad &  \quad 0&  0\qquad\\
    0  \quad &  \quad 0 \quad  &  \quad \cdots  \quad &  \quad
    \lambda_{n-1}
    I_{k_{n-1}}&
    0\qquad\\
    0  \quad &  \quad 0 \quad  &  \quad \cdots  \quad & \quad  0&
    \quad \lambda_n I_{k_{n }}
\end{aligned} \right )U^*\ | \ \text {$U$  is in $\mathcal
U_k $} \right \}.
$$
By Corollary 12 in \cite{Sza} or Theorem 3 in \cite{DH}, the
covering number of $\Omega(k_1,\ldots,k_{n-1},k_{n })$ by
$\omega$-$\|\cdot\|$-balls in $\mathcal M_k(\Bbb C)$ is upperbounded
by
$$
\nu_\infty(\Omega(k_1,\ldots,k_{n-1},k_{n }), \omega)\le\left (
\frac {C_2}{\omega} \right )^{k^2-\sum_{i=1}^nk_i^2},
$$ where $C_2$ is a   constant which does not depend on $k,
k_1,\ldots,k_n$ (may depend on $n$ and $\|x\|$).

 Let $\mathcal I$ be the set consisting of all these
$(k_1,\ldots, k_n)$ in $\Bbb Z^n$ such that  $1\le k_1,\ldots,
k_{n}\le k$ and $k_1+\cdots +k_n=k$. Then the cardinality of the set
$\mathcal I$ is equal to
$$
 \frac {(k-1) !}{(n-1)! (k-n)!}.
$$
Note that
$$
\sum_{i=1}^nk_i^2 \ge k^2/n
$$ for all $1\le k_1,\ldots, k_n\le k$ with $k_1+\cdots
+k_n=k$; and by ($**$)
$$
\Gamma^{(top)}_R(x; k,\epsilon, P_1,\ldots, P_r)$$ is contained in $
2\omega$-neighborhood of the set$$  \bigcup_{(k_1,\ldots, k_{n})\in
\mathcal I} \ \ \Omega(k_1,\ldots,k_{n}).
$$ It follows that  the covering number of the set
$$
\Gamma^{(top)}_R(x; k,\epsilon, P_1,\ldots, P_r)
$$  by $3\omega$-$\|\cdot\|$-balls in $\mathcal M_k(\Bbb C)$ is upperbounded by
$$
\nu_\infty(\Gamma^{(top)}_R(x; k,\epsilon, P_1,\ldots,
P_r),3\omega)\le \frac {(k-1) !}{(n-1)! (k-n)!} \cdot \left ( \frac
{C_2}{\omega} \right )^{k^2-k^2/n}.
$$
Thus
$$
\delta_{top}(x) \le \limsup_{\omega\rightarrow 0^+}
\limsup_{k\rightarrow\infty} \frac {\log\left (\frac  {(k-1)
!}{(n-1)! (k-n)!}\cdot \left ( \frac {C_2}{\omega} \right
)^{k^2-k^2/n} \right )}{-k^2\log(3\omega)}=1-\frac 1 n.
$$

\end{proof}

\subsection{Lower-bound}

We follow the notation from last subsection.

\begin{proposition}
  Suppose that $x$ is a self-adjoint element with the finite spectrum $\sigma(x)$ in
  $\mathcal A$. Then
$$
  \delta_{top}(x) \ge 1- \frac 1 n,
$$ where $n$ is the cardinality of the set $ \sigma(x)$.
\end{proposition}

\begin{proof}
Suppose that   $\lambda_1, \ldots, \lambda_n$ are distinct spectrum
of $x$. There is some positive number $\theta$ such that
$$
|\lambda_i-\lambda_j|>\theta, \qquad \forall \ \ 1\le i\ne j \le n.
$$
Assume $k=nt$ for some positive integer $t$. Let $$ A_k=diag
(\lambda_1 I_t,\ldots, \lambda_n I_t)
$$ be a diagonal matrix in $\mathcal M_k(\Bbb C)$ where $I_t$ is the
$t\times t$ identity matrix. It is easy to see that, for all
$R>\|x\|$, $r\ge 1$ and $\epsilon>0$, we have
$$
A_k\in \Gamma^{(top)}_R(x; k,\epsilon, P_1,\ldots, P_r).
$$ For any $\omega>0$, applying Theorem 3.2 for $n=m$ and
$\delta=\frac 1 2\omega $, we have
$$\begin{aligned}
\nu_{\infty}( \Gamma^{(top)}_R(x; k,\epsilon, P_1,\ldots,
P_r),\omega)&\ge (C_1\cdot 8\delta/\theta)^{-k^2}\cdot \left
(\frac{C\theta}{ \delta}\right )^{-mt^2} \\ &=
(16C_1\omega/\theta)^{-k^2}\cdot \left (\frac{2C\theta}{  \omega
}\right )^{-mt^2}.
\end{aligned}$$
Note that $k=nt=mt$ and $\theta$ is some fixed number. A quick
computation shows that
$$
\delta_{top}(x) \ge 1- \frac 1 n.
$$
\end{proof}

\begin{proposition}
 Suppose that $x$ is a self-adjoint element in $\mathcal A$ with
 infinite spectrum. Then
  $$ \delta_{top}(x) \ge 1.$$
\end{proposition}
\begin{proof}
For any $0<\theta<1$, there are $\lambda_1,\ldots, \lambda_m$ in the
spectrum of $x$, $\sigma(x)$, satisfying (i)
$$
|\lambda_i-\lambda_j|\ge \theta; $$  and (ii) for any $\lambda$ in
$\sigma(x)$, there is some $\lambda_j$ with $|\lambda-\lambda_j|\le
\theta$. By functional calculus, for any $R>\|x\|$, $r\ge 1$ and
$\epsilon>0$, there are some positive integer $n\ge m$ and real
numbers $\lambda_{m+1}, \ldots, \lambda_n$ in $\sigma(x)$
satisfying: for every $t\ge 1$  the matrix
$$A=
diag(\lambda_1 I_t, \lambda_2 I_t, \ldots, \lambda_m
I_t,\lambda_{m+1},\ldots, \lambda_n)
$$ is in
$$
\Gamma^{(top)}_R(x;k,\epsilon, P_1,\ldots, P_r),
$$ where  we assume that $k=mt+n-m$.
  For any $\omega>0,$ let $\delta=\frac 1 2 \omega$. By Theorem 3.2, we know that
$$
\nu_{\infty}( \Gamma^{(top)}_R(x; k,\epsilon, P_1,\ldots,
P_r),\omega) \ge (C_1\cdot 4\delta/\theta)^{-k^2}\cdot \left
(\frac{C\theta}{ \delta}\right )^{-(2mt^2+4m(n-m)t+2(n-m)^2)}.
$$
Thus
$$
\limsup_{k\rightarrow\infty} \frac {\log(\nu_{\infty}(
\Gamma^{(top)}_R(x; k,\epsilon, P_1,\ldots,
P_r),\omega))}{-k^2\log\omega}\ge \frac {\log( \frac {4C_1}
2)-\log\theta}{\log\omega}+1 +\frac 2 m  \frac { \log (2C)+ \log
\theta  }{\log\omega} -\frac 2 m.
$$
Then,
$$
\delta_{top}(x) \ge 1- \frac 2 m.
$$
When $\theta$ goes to $0$, $m$ goes to infinity as $\sigma(x)$ has
infinitely many elements. Therefore,
$$
\delta_{top}(x) \ge 1.
$$
\end{proof}

\subsection{Topological free entropy dimension in one variable case} By
Proposition 4.1, Proposition 4.2 and Proposition 4.3, we have the
following result.

\begin{theorem}
Suppose $x$ is a self-adjoint element in a unital  C$^*$ algebra
$A$. Then
$$
\delta_{top}(x) = 1- \frac 1 n,
$$ where $n$ is the cardinality of the set  $\sigma(x)$ and $\sigma(x) $ is  the set of spectrum of $x$
in $\mathcal A$. Here we assume that $\frac 1 \infty=0$.
\end{theorem}

\section{Topological free entropy dimension of $n$-tuple in   unital C$^*$ algebras  }

%\subsection{Lower-bound of topological free entropy dimension of an infinite dimensional C$^*$ algebra}
%\begin{definition}
%Suppose $\mathcal A$ is a C$^*$ algebra and $x_1,\ldots, x_n$ is a
%family of self-adjoint elements of $\mathcal A$ that generates
%$\mathcal A$ as a C$^*$ algebra. If for any $r>0$, $\epsilon>0$ and
%$R>\max\{x_1,\ldots,x_n,y_1,\ldots,y_m\}$, there is some positive
%integer $k_0$ such that
%$$
%\Gamma^{(top)}_R(x_1,\ldots, x_n: y_1,\ldots, y_m; k,\epsilon, P_1,\ldots,
%P_r) \ne \varnothing, \qquad \forall \   k\ge k_0
%$$ then $\mathcal A$ is called having approximation property.

%\end{definition}

%\begin{proposition}
%Suppose that $\mathcal A$ is an infinite dimensional C$^*$ algebra
%with approximation property. Then
%$$
%\delta(x_1,\ldots, x_n)\ge 1,
%$$ for any family of self-adjoint elements $x_1,\ldots, x_n$ that
%generates $\mathcal A$.
%\end{proposition}
%\begin{proof}
%Since $x_1,\ldots, x_n$ generates an infinite dimensional C$^*$
%algebra $\mathcal A$, at least one of $x_1,\ldots, x_n$ has infinite
%spectrum. Without loss of generality, we assume that $\sigma(x_1)$
%is a set with infinite cardinality. From the definition of
%topological free entropy dimension it is not hard to see   that
%$$
%\delta(x_1,\ldots, x_n)\ge \delta(x_1:x_1,\ldots, x_n).
%$$ By Theorem 3, we have
%$$
%\delta(x_1,\ldots, x_n)\ge 1.
%$$
%\end{proof}

\subsection{An equivalent definition of topological free entropy dimension}
Suppose that $\mathcal A$ is a unital C$^*$ algebra and $x_1,\ldots,
x_n,y_1,\ldots, y_m$ are self-adjoint elements in $\mathcal A$. For
every $R,\epsilon>0$ and positive integers $r, k$, let
$$
\Gamma^{(top)}_R(x_1,\ldots, x_n:y_1,\ldots, y_m; k,\epsilon,
P_1,\ldots, P_r)
$$ be Voiculescu's norm-microstate space defined in section 2.4.

Define
$$
\nu_2(\Gamma^{(top)}_R(x_1,\ldots, x_n: y_1,\ldots, y_m; k,\epsilon,
P_1,\ldots, P_r),\omega)
$$ to be the covering number of the set $
\Gamma^{(top)}_R(x_1,\ldots, x_n: y_1,\ldots, y_m; k,\epsilon,
P_1,\ldots, P_r) $ by $\omega$-$\| \cdot\|_2$-balls in the metric
space $(\mathcal M_k^{s.a}(\Bbb C))^n$ equipped with trace norm (see
Definition 2.2).
\begin{definition}
Define
$$
  \begin{aligned}
     \tilde\delta_{top}(x_1,\ldots, & \ x_n: y_1,\ldots, y_m;  \omega)\\
     & =
     \sup_{R>0} \ \inf_{\epsilon>0, r\in \Bbb N}  \ \limsup_{k\rightarrow\infty} \frac {\log(\nu_2(\Gamma^{(top)}_R(x_1,
     \ldots, x_n: y_1,\ldots, y_m; k,\epsilon, P_1,\ldots,
P_r),\omega))}{-k^2\log\omega}\\ \end{aligned}
$$
And
$$ \tilde\delta_{top}(x_1,\ldots, \ x_n: y_1,\ldots, y_m)=
 \limsup_{\omega\rightarrow 0^+}  \tilde\delta_{top}(x_1,\ldots,   x_n: y_1,\ldots, y_m;  \omega)
$$
\end{definition}

 The following proposition was pointed out by
Voiculescu in \cite{Voi}. For the sake of completeness, we also
include a proof here.
\begin{proposition}Suppose that $\mathcal A$ is a unital C$^*$ algebra and
$x_1,\ldots, x_n, y_1,\ldots, y_m$ are self-adjoint elements in
$\mathcal A$. Then
$$
 \tilde\delta_{top}(x_1,\ldots, \ x_n: y_1,\ldots, y_m)= \delta_{top}(x_1,\ldots, \ x_n: y_1,\ldots,
 y_m),
$$ where $\delta_{top}(x_1,\ldots, \ x_n: y_1,\ldots,
 y_m)$ is the topological free entropy dimension of $x_1,\ldots, x_n$ in
 presence of $y_1,\ldots,y_m$.
\end{proposition}
\begin{proof}
  This is an easy consequence of Lemma 1 in \cite {Sza}. Let
  $\lambda$ be the Lebesgue measure on $(\mathcal M_k^{s.a}(\Bbb
  C))^n$. Let, for every $\omega>0$,
  $$
     \begin{aligned}
    &        B_\infty(\omega)= \{(A_1,\ldots, A_n) \in (\mathcal M_k^{s.a}(\Bbb
  C))^n \ | \ \| (A_1,\ldots, A_n)  \|\le \omega\}\\
   &        B_2(\omega)= \{(A_1,\ldots, A_n) \in (\mathcal M_k^{s.a}(\Bbb
  C))^n \ | \ \| (A_1,\ldots, A_n)  \|_2\le  \omega\}
     \end{aligned}
  $$
  It follows from the results in \cite{Sza} or  Theorem 8 in
  \cite{DH} that, for some $M_1,M_2$ independent of $k, \omega$ such
  that
 \begin{equation}
 {\lambda(B_\infty(1))}\le {\lambda(B_\infty (\omega/4))}  \left (\frac {M_1}
    \omega \right)^{nk^2} \quad \text { and } \quad  \left (\frac {M_2}
    {2\sqrt n\omega} \right)^{nk^2}   {\lambda(B_2(2 \sqrt
    n\omega))}\le  {\lambda(B_2(1))} .
 \end{equation}
 For every $\omega>0$ and any subset set $K$ of $(\mathcal M_k^{s.a}(\Bbb
  C))^n$  , let
  $$
     \begin{aligned}
       &K( \omega, \|\cdot\|) =  \{(A_1,\ldots, A_n) \in (\mathcal M_k^{s.a}(\Bbb
  C))^n \ | \ \| (A_1,\ldots, A_n)-(D_1,\ldots,D_n)  \|\le  \omega \\ & \qquad \qquad  \qquad \qquad
   \qquad \qquad  \qquad \qquad   \qquad \qquad \text { for some } (D_1,\ldots,D_n) \in K\}\\
  &K(\omega, \|\cdot\|_2) =  \{(A_1,\ldots, A_n) \in (\mathcal M_k^{s.a}(\Bbb
  C))^n \ | \ \| (A_1,\ldots, A_n)-(D_1,\ldots,D_n)  \|_2\le \omega \\ & \qquad \qquad
   \qquad \qquad  \qquad \qquad  \qquad \qquad  \qquad \qquad \text { for some } (D_1,\ldots,D_n) \in K\}\\
     \end{aligned}
  $$
  Note the following fact:
$$
\| (A_1,\ldots, A_n)   \|_2\le \sqrt n \| (A_1,\ldots, A_n)\|,
 \qquad \forall \ (A_1,\ldots, A_n) \in (\mathcal M_k^{s.a}(\Bbb
  C))^n. $$
  It follows from Lemma 1 in \cite{Sza}   that
 $$
     \nu_\infty(K,\omega) \le \frac{\lambda
    (K(\omega,\|\cdot\|))}{\lambda(B_\infty (\omega/4))} ;
  $$ and
  $$
     \frac{\lambda
    (K( { \sqrt n\omega} ,\|\cdot\|_2))}{\lambda(B_2(2 \sqrt n\omega))}\le
    \nu_2(K( { \sqrt n\omega}  ,\|\cdot\|_2), 2 \sqrt n\omega) .
  $$
Combining with the equalities (5.1.1), we get
  $$
     \nu_\infty(K,\omega) \le \frac{\lambda
    (K(\omega,\|\cdot\|))}{\lambda(B_\infty (\omega/4))}\le \left (\frac {M_1}
    \omega \right)^{nk^2} \frac{\lambda
    (K(\omega,\|\cdot\|))}{\lambda(B_\infty(1))}\le\left (\frac {M_1}
    \omega \right)^{nk^2} \frac{\lambda
    (K( { \sqrt n\omega}  ,\|\cdot\|_2))}{\lambda(B_\infty(1))};
  $$ and
  $$
   \left (\frac {M_2}
    {2\sqrt n\omega} \right)^{nk^2} \frac{\lambda
    (K( { \sqrt n\omega} ,\|\cdot\|_2))}{\lambda(B_2(1))}\le \frac{\lambda
    (K( { \sqrt n\omega} ,\|\cdot\|_2))}{\lambda(B_2(2 \sqrt n\omega))}\le
    \nu_2(K( { \sqrt n\omega}  ,\|\cdot\|_2), 2 \sqrt n\omega)\le  \nu_2(K,  \sqrt n\omega).
  $$
  Therefore, we have
  $$
 \nu_2(K,  \sqrt n\omega  )\le \nu_\infty(K,\omega)\le \left (\frac{2\sqrt n M_1}{M_2}\right)^{nk^2}\frac {\lambda(B_2(1))}
 {\lambda(B_\infty(1))}\nu_2(K,
  \sqrt n\omega).
  $$
  It is a well-known fact (for example see Theorem 8 in \cite{DH}) that
  $$
\frac {\lambda(B_2(1))}{\lambda(B_\infty(1))}\le C_3^{nk^2}
  $$ for some universal constant $C_3>0$.
  Hence
   $$
 \nu_2(K,  \sqrt n \omega )\le \nu_\infty(K,\omega)\le  \left (\frac{ 2\sqrt n M_1C_3}{M_2}\right)^{nk^2}\nu_2(K,
 \sqrt  n\omega).
  $$
Let $K$ be $\Gamma^{(top)}_R(x_1,
     \ldots,x_n: y_1,\ldots, y_m; k,\epsilon, P_1,\ldots,
P_r)$. By the definitions of $\tilde\delta_{top}$ and
$\delta_{top}$, we have
$$
 \tilde\delta_{top}(x_1,\ldots, \ x_n: y_1,\ldots, y_m)= \delta_{top}(x_1,\ldots, \ x_n: y_1,\ldots,
 y_m).
$$

\end{proof}

\subsection{Upper-bound of topological free entropy dimension in a unital  C$^*$ algebra}

Let us recall Voiculescu's definition of free dimension capacity in
\cite{Voi}.
\begin{definition}
Suppose that $\mathcal A$ is a unital C$^*$ algebra with a family of
self-adjoint generators $x_1,\ldots, x_n$. Suppose that $TS(\mathcal
A)$ is the set consisting of all tracial states of $\mathcal A$. If
$TS(\mathcal A)\ne \varnothing$, define {\em Voiculescu's free
dimension capacity} $\kappa\delta(x_1,\ldots, x_n)$ of
$x_1,\ldots,x_n$ as follows,
$$
\kappa\delta(x_1,\ldots, x_n) = \sup_{\tau\in TS(\mathcal A)}
\delta_0(x_1,\ldots, x_n:\tau),
$$ where $
\delta_0(x_1,\ldots, x_n:\tau)$ is Voiculescu's (von Neumann
algebra) free entropy dimension of $x_1,\ldots, x_n$ in $\langle
\mathcal A,\tau\rangle$.
\end{definition}

The relationship between topological free entropy dimension of a
unital C$^*$ algebra with a unique tracial state and its free
dimension capacity is indicated by the following result.
\begin{theorem}
Suppose that $\mathcal A$ is a unital C$^*$ algebra with a family of
self-adjoint generators $x_1,\ldots, x_n$. Suppose that $TS(\mathcal
A)$ is the set consisting of all tracial states of $\mathcal A$. If
$TS(\mathcal A)$ is a set with a single element, then
$$
\delta_{top}(x_1,\ldots, x_n) \le \kappa\delta(x_1,\ldots, x_n) .
$$
\end{theorem}
To prove the preceding theorem, we need the following lemma.

\begin{sublemma}
Suppose that $\mathcal A$ is a unital C$^*$ algebra with a family of
self-adjoint generators $x_1,\ldots, x_n$. Suppose that $TS(\mathcal
A)\ne \varnothing$ is the set consisting of all tracial states of
$\mathcal A$. Let $R>\max\{\|x_1\|,\ldots,\|x_n\|\}$ be some
positive number. Then for any $m\ge 1$, there is some $r_m\in \Bbb
N$ such that
$$
\Gamma^{(top)}_R(x_1,\ldots, x_n; k,\frac 1 {r_m}, P_1,\ldots,
P_{r_m}) \subseteq \cup_{\tau\in TS(\mathcal
A)}\Gamma_R(x_1,\ldots,x_n; k,m,\frac 1m;\tau),\quad \forall \ k\ge
1
$$
where $\Gamma_R(x_1,\ldots,x_n; k,m,\frac 1m;\tau)$ is microstate
space of $x_1,\ldots, x_n$ with respect to $\tau$ (see \cite{V2}).
\end{sublemma}

%\begin{lemma} Suppose that $\mathcal A$ is a unital C$^*$ algebra
%with a family of self-adjoint generators $x_1,\ldots, x_n$. Suppose
%that $TS(\mathcal A)\ne \varnothing$ is the space consisting all
%tracial states of $\mathcal A$. Let
%$R>\max\{\|x_1\|,\ldots,\|x_n\|\}$ be some positive number. Then
%there is some $\tau$ in $TS(\mathcal A)$ satisfying,  for any $m\ge
%1$, there is some $r\in \Bbb N$ such that
%$$
%\Gamma^{(top)}_R(x_1,\ldots, x_n; k,\frac 1 r, P_1,\ldots, P_r)
%\subseteq \Gamma_R(x_1,\ldots,x_n; k,m,\frac 1m;\tau),\quad \forall
%\ k\ge 1.
%$$
%\end{lemma}

\begin{proof}[Proof of Sublemma 5.2.1: ] We will prove the result by
contradiction. Suppose, to the contrary, there is some $m_0\ge 1$ so
that following holds: for any $r\in \Bbb N$, there are some $k_r\ge
1$ and some
$$
(A_1^{(r)}, A_2^{(r)},\ldots, A_n^{(r)})\in
\Gamma^{(top)}_R(x_1,\ldots, x_n; k_r,\frac 1 {r}, P_1,\ldots,
P_{r})
$$
satisfying \begin{equation}(A_1^{(r)}, A_2^{(r)},\ldots,
A_n^{(r)})\notin \cup_{\tau\in TS(\mathcal
A)}\Gamma_R(x_1,\ldots,x_n; k_r,m,\frac 1m;\tau).
\end{equation}
Let $\alpha$ be a free ultrafilter in $\beta(\Bbb N)\setminus \Bbb
N$. Let $\mathcal N=\prod_{r=1}^\alpha \mathcal M_{k_r}(\Bbb C)$ be
the von Neumann algebra ultraproduct  of $\{\mathcal M_{k_r}(\Bbb
C)\}_{r=1}^\infty$ along the ultrafilter $\alpha$, i.e.
$\prod_{r=1}^\alpha \mathcal M_{k_r}(\Bbb C)$ is the quotient
algebra of the C$^*$ algebra $\prod_{r=1}^\infty \mathcal
M_{k_r}(\Bbb C)$ by $\mathcal I_2$, the $0$-ideal of the trace
$\tau_{\alpha}$, where $ \tau_{\alpha} ((A_{ r
})_{r=1}^\infty)=\lim_{r\rightarrow \alpha}\frac{Tr(A_{r})}{k_r}$.
Let, for each $1\le j\le n$, $a_j=[( A_r^{(j)} )_{r=1}^\infty]$ be a
self-adjoint element
  in $\mathcal N$. By mapping $x_j$ to $a_j$, there is a
 unital $*$-homomorphism $\psi$ from the C$^*$  algebra $\mathcal A$  onto the C$^*$ subalgebra generated
 by $\{a_1,\ldots,a_n\}$ in $\mathcal N$.

 Let $\tau_0$ be the tracial state on $\mathcal A$ which is induced
 by $\tau_\alpha$ on $\psi(\mathcal A)$, i.e.
 $$
    \tau_0(x) = \tau_\alpha(\psi(x)), \qquad   \forall \ x\in
    \mathcal A.
 $$
It follows that when $r$ is large enough,
$$(A_1^{(r)}, A_2^{(r)},\ldots, A_n^{(r)})\in
\Gamma_R(x_1,\ldots,x_n; k_r,m,\frac 1m;\tau_0),$$ which contradicts
with  the inequality (5.2.1). This complete the proof.
\end{proof}

 \begin{proof}[Proof of Theorem 5.1: ]  Let
 $R>\max\{\|x_1\|,\ldots,\|x_n\|\}$. Let $\tau$ be the unique trace of $\mathcal A$. By Sublemma 5.2.1, for any $m\ge 1$, there is $r\in \Bbb N$ such
 that
$$
\Gamma^{(top)}_R(x_1,\ldots, x_n; k,\frac 1 r, P_1,\ldots, P_r)
\subseteq \Gamma_R(x_1,\ldots,x_n; k,m,\frac 1m;\tau),\quad \forall
\ k\ge 1.
$$
Therefore, for any $1>\omega>0$, we have
$$
\nu_2(\Gamma^{(top)}_R(x_1,\ldots, x_n; k,\frac 1 r, P_1,\ldots,
P_r),\omega) \le \nu_2( \Gamma_R(x_1,\ldots,x_n; k,m,\frac
1m;\tau),\omega),\quad \forall \ k\ge 1.
$$
  Now it is easy to check that
$$
\tilde\delta_{top}(x_1,\ldots,x_n) \le
\delta(x_1,\ldots,x_n;\tau)=\kappa\delta(x_1,\ldots, x_n) .
$$By Proposition 5.1, we know that
$$\delta_{top}(x_1,\ldots, x_n) \le \kappa\delta(x_1,\ldots, x_n) .
$$
\end{proof}

\begin{remark}
Combining Theorem 5.1 with the results in \cite{HaSh} or \cite
{Jung2}, we will be able to compute the upper-bound of topological
free entropy dimension for a large class of unital C$^*$ algebras.
For example, $\delta_{top}(x_1,\ldots,x_n)\le 1$ if $x_1,\ldots,
x_n$ is a family of self-adjoint operators that generates an
irrational rotation algebra $\mathcal A$.
\end{remark}

\subsection{Lower-bound of topological free entropy dimension in a unital  C$^*$ algebra}

In this subsection, we assume that $\mathcal A$ is a finitely
generated, infinite dimensional, unital simple C$^*$ algebra  with a
unique tracial state $\tau$. Assume that $x_1,\ldots, x_n$ is a
family of self-adjoint generators of $\mathcal A$. Let $H$ be the
Hilbert space $L^2(\mathcal A,\tau)$. Without loss of generality, we
might assume that $\mathcal A$ is faithfully represented on the
Hilbert space $H$. Let $\mathcal M$ be the von Neumann algebra
generated by $\mathcal A$ on $H$. It is not hard to see that
$\mathcal M$ is a diffuse von Neumann algebra with a tracial state
$\tau$.

For each positive integer $m$, there is a family of mutually
orthogonal projections $p_1,\ldots, p_m$ in $\mathcal M$ such that
$\tau(p_j)=1/m$ for $1\le j\le m$. Let $$y_m=1\cdot  p_1+ 2\cdot
p_2+\cdots+ m\cdot p_m=\sum_{j=1}^m j\cdot p_j\ \in \mathcal M.$$

Let $\{P_r(x_1,\ldots,x_n)\}_{r=1}^\infty$ be defined as in section
2.3. Thus $\{P_r(x_1,\ldots,x_n)\}_{r=1}^\infty$ is dense in
$\mathcal M$ with respect to the strong operator topology. Hence,
for each $m\ge 1$, there is some self-adjoint element
$P_{r_m}(x_1,\ldots, x_n)$ in $\mathcal A$ such that
$$
\|y_m-P_{r_m}(x_1,\ldots, x_n)\|_2\le \frac 1 {m^3},
$$ where $\|a\|_2= \sqrt{\tau(a^*a)}$ for all $a\in \mathcal M$.

\begin{lemma}
Let $\mathcal A$ be finitely generated, infinite dimensional, unital
simple C$^*$ algebra  with a unique tracial state $\tau$. Assume
that $x_1,\ldots, x_n$ is a family of self-adjoint generators of
$\mathcal A$. Let $H$, $\mathcal M$ be defined as above. For each
$m\ge 1$, let $y_m$ and $P_{r_m}(x_1,\ldots, x_n)$ be chosen as
above. Then
$$
\delta_{top}(x_1,\ldots, x_n) \ge \delta_{top}(P_{r_m}(x_1,\ldots,
x_n):x_1,\ldots,x_n).
$$
\end{lemma}

\begin{proof}
Let $R>\max\{\| P_{r_m}(x_1,\ldots, x_n) \|, \|x_1\|, \ldots,
\|x_n\|\}$. There exists a positive constant $D>1$ such that
$$
\|P_{r_m}(A_1,\ldots,A_n) -P_{r_m}(B_1,\ldots,B_m)\|\le
D\|(A_1,\ldots,A_n) -(B_1,\ldots,B_m)\|
$$ for all $A_1,\ldots,A_n,B_1,\ldots,B_n$ in $\mathcal M_k(\Bbb C)$ satisfying $0\le
\|A_1\|,\ldots,\|A_n\|,\|B_1\|,\ldots,\|B_n\|\le R$.

Then it is not hard to verify that, for $\omega>0$,
$$\begin{aligned}
\nu_\infty(\Gamma^{(top)}_R (P_{r_m}(x_1,\ldots, x_n)&:x_1,\ldots,x_n; k, \epsilon, P_1,\ldots, P_r), \omega)\\
&\le \nu_\infty(\Gamma^{(top)}_R(x_1,\ldots,x_n; k,\epsilon,
P_1,\ldots, P_r), \frac {\omega} {4D})\end{aligned}
$$ for each $r\ge r_m$ and $\epsilon<\frac \omega 4$. By definition
of $\delta_{top}$ and Remark 2.3, we have
$$
  \delta_{top}(P_{r_m}(x_1,\ldots,
x_n):x_1,\ldots,x_n) \le \delta_{top}(x_1,\ldots, x_n).
$$
\end{proof}

\begin{definition}
Suppose $\mathcal A$ is a unital C$^*$ algebra and $x_1,\ldots, x_n$
is a family of self-adjoint elements of $\mathcal A$ that generates
$\mathcal A$ as a C$^*$ algebra. If for any
$R>\max\{\|x_1\|,\ldots,\|x_n\|,\|y_1\|,\ldots,\|y_m\|\}$, $r>0$,
$\epsilon>0$, there is a sequence of positive integers
$k_1<k_2<\cdots $ such that
$$
\Gamma^{(top)}_R(x_1,\ldots, x_n: y_1,\ldots, y_m; k_s,\epsilon,
P_1,\ldots, P_r) \ne \varnothing, \qquad \forall \   s\ge 1
$$ then $\mathcal A$ is called having approximation property.

\end{definition}
\begin{lemma}
 Let $\mathcal A$ be a finitely generated, infinite dimensional,
unital simple C$^*$ algebra  with a unique tracial state $\tau$.
Assume that $A$ has approximation property.  Assume that
$x_1,\ldots, x_n$ is a family of self-adjoint generators of
$\mathcal A$. Let $H$, $\mathcal M$ be defined as above. Let $m$ be
a positive integer. Let $y_m$ and $P_{r_m}(x_1,\ldots, x_n)$ be
chosen as above. Let $R>\max\{\| P_{r_m}(x_1,\ldots, x_n) \|,
\|x_1\|, \ldots, \|x_n\|\}$. Then there is some positive integer
$r>r_m$ so that the following hold: $\forall \ k\ge 1 $, if
$$
(B,A_1,\ldots,A_n)\in \Gamma^{(top)}_R(P_{r_m}(x_1,\ldots,
x_n),x_1,\ldots,x_n; k, \frac 1 r, P_1,\ldots, P_r),
$$ then there are some $1\le k_1,\ldots,k_m\le k$ with $\frac 1 m -\frac 1 r\le \frac {k_j} k \le \frac 1 m +\frac 1
  r  \text {  for each } \ 1\le j\le m  $ and $k_1+\cdots+k_m=k$, and a unitary matrix
$U$ in $\mathcal U(k)$ satisfying
 $$
    \|B- U \left (
          \begin{aligned}
              \begin{aligned}
    1 \cdot I_{k_1} \quad  & \quad 0  \quad & \quad  \cdots  \quad  &  \quad   0 \\
    0 \quad  & \quad   2 \cdot I_{k_2} \quad  & \quad  \cdots  \quad  &  \quad 0 \\
      \cdots \quad& \quad \cdots & \quad  \ddots\quad & \quad\cdots\\
    0  \quad &  \quad 0 \quad  &  \quad \cdots  \quad  &
    \quad   m\cdot I_{k_{m}}
\end{aligned}
          \end{aligned}
      \right ) U^*\|_2\le \frac 2{m^3}.
  $$

\end{lemma}
\begin{proof}
We will prove the result by contradiction. Assume, to the contrary,
 for all $r\ge r_m$ there are some
$k_r\ge 1$ and some
$$
(B^{(r)},A_1^{(r)},\ldots,A_n^{(r)})\in
\Gamma^{(top)}_R(P_{r_m}(x_1,\ldots, x_n),x_1,\ldots,x_n;
k_{r},\frac 1 r, P_1,\ldots, P_r),
$$ satisfying
\begin{equation}
    \|B^{(r)}- U \left (
          \begin{aligned}
              \begin{aligned}
    1 \cdot I_{s_1} \quad  & \quad 0  \quad & \quad  \cdots  \quad  &  \quad   0 \\
    0 \quad  & \quad   2 \cdot I_{s_2} \quad  & \quad  \cdots  \quad  &  \quad 0 \\
      \cdots \quad& \quad \cdots & \quad  \ddots\quad & \quad\cdots\\
    0  \quad &  \quad 0 \quad  &  \quad \cdots  \quad  &
    \quad   m\cdot I_{s_{m}}
\end{aligned}
          \end{aligned}
      \right ) U^*\|_2> \frac 2{m^3},
 \end{equation}
  for all $1\le s_1,\ldots,s_m\le k_r$ with $\frac 1 m -\frac 1 r\le \frac {s_j} {k_r} \le \frac 1 m +\frac 1
  r  \text {  for each } \ 1\le j\le n  $ and $s_1+\cdots+s_m=k_r$, and all unitary matrix
$U$ in $\mathcal U(k)$.

Let $\alpha$ be a free ultrafilter in $\beta(\Bbb N)\setminus \Bbb
N$. Let $\mathcal N=\prod_{r=1}^\alpha \mathcal M_{k_r}(\Bbb C)$ be
the von Neumann algebra ultraproduct  of $\{\mathcal M_{k_r}(\Bbb
C)\}_{r=1}^\infty$ along the ultrafilter $\alpha$, i.e.
$\prod_{r=1}^\alpha \mathcal M_{k_r}(\Bbb C)$ is the quotient of the
C$^*$ algebra $\prod_{r=1}^\infty \mathcal M_{k_r}(\Bbb C)$ by
$\mathcal I_2$, the $0$-ideal of the trace $\tau_{\alpha}$, where $
\tau_{\alpha} ((A_r)_{r=1}^\infty)=\lim_{r\rightarrow
\alpha}\frac{Tr(A_r)}{k_r}$. Let, for each $1\le j\le n$, $a_j=[(
A_r^{(j)} )_{r=1}^\infty]$ be a self-adjoint element
  in $\mathcal N$. By mapping $x_j$ to $a_j$, there is a
 unital $*$-homomorphism $\psi$ from the C$^*$  algebra $\mathcal A$  onto the C$^*$ subalgebra generated
 by $\{a_1,\ldots,a_n\}$ in $\mathcal N$. Since $\mathcal A$ is a
 simple C$^*$ algebra and $\psi(I_{\mathcal A})=I_{\mathcal N}$,
 $\psi$ actually is a $*$-isomorphism. Since   $\mathcal A$ has a unique trace $\tau$, $\psi$ induces a $*$-isomorphism (still denoted by $\psi$) from
 $\mathcal M$ onto the von Neumann subalgebra generated by
 $a_1,\ldots, a_n$ in $\mathcal N$. Therefore,
 $$
\|y_m-P_{r_m}(x_1,\ldots,x_n)\|_2=\|\psi(y_m)-P_{r_m}(a_1,\ldots,a_n)\|_{2,\tau_\alpha}\le
\frac 1 {m^3}.
 $$
 This contradicts with the definition of $y_m$ and inequality
 (5.3.1).
\end{proof}

The following lemma is well-known (for example, see Lemma 4.1 in
\cite{V2}).

\begin{lemma}
Suppose $A$, or $B$, is a  self-adjoint matrix in $\mathcal
M_k^{s.a.}(\Bbb C)$ with a list of eigenvalues $\lambda_1\le
\lambda_2\le \cdots\le \lambda_k$, or $\mu_1\le
\mu_2\le\cdots\le\mu_k$ respectively. Then
$$
\sum_{j=1}^k |\lambda_j-\mu_j|^2\le Tr ((A-UBU^*)^2),
$$ where $U$ is any unitary matrix in $\mathcal U(k)$.
\end{lemma}

\begin{lemma}Let $r,m$ be some positive integer with $4<m<r$.
  Suppose $k_1,\ldots,k_m$ is a family of positive integers such
  that $\frac 1 m-\frac 1 r\le \frac {k_j} k\le \frac 1 m+\frac 1 r$
  for all $1\le j\le k$ and $k_1+\cdots+k_m=k$. If $A$ is a
  self-adjoint matrix in $\mathcal M_k(\Bbb C)$ such that, for
  some unitary matrix $U$ in $\mathcal U(k)$,
 $$ \|A- U \left (
          \begin{aligned}
              \begin{aligned}
    1 \cdot I_{k_1} \quad  & \quad 0  \quad & \quad  \cdots  \quad  &  \quad   0 \\
    0 \quad  & \quad   2 \cdot I_{k_2} \quad  & \quad  \cdots  \quad  &  \quad 0 \\
      \cdots \quad& \quad \cdots & \quad  \ddots\quad & \quad\cdots\\
    0  \quad &  \quad 0 \quad  &  \quad \cdots  \quad  &
    \quad   m\cdot I_{k_{m}}
\end{aligned}
          \end{aligned}
      \right ) U^*\|_2\le \frac 2{m^3},
$$  then, for any $ \omega>0$  we have
$$
  \nu_2(\Omega(A),\omega)\ge (8C_1\omega)^{-k^2}\cdot \left (\frac{2C }{ \omega}\right
)^{\frac{-56 k^2}{m}}
$$ for some   constants $C_1, C>1$ independent of $k,\omega$, where
$$
\Omega(A) =\{W^*AW \ | \ W \in \mathcal U(k)\}.
$$

\end{lemma}

\begin{proof}
Suppose that $\lambda_1\le \lambda_2\le \ldots\le \lambda_k$ are the
eigenvalues of $A$. For each $1\le j\le m$, let
$$
T_j= \{i \in \Bbb N\ | \ (\sum_{t=0}^{j-1}k_{t})+1\le i\le
\sum_{t=0}^{j}k_{t} \text { and } |\lambda_i-j|\le \frac 1 m \}$$
and $$\hat
T_j=\{(\sum_{t=0}^{j-1}k_{t})+1,(\sum_{t=0}^{j-1}k_{t})+2,\cdots,
\sum_{t=0}^{j}k_{t}  \}\setminus T_j,
$$ here we assume that $k_0=0$. Let $B= diag(1\cdot I_{k_1},\cdots, m\cdot I_{k_m})$ be a diagonal matrix in $\mathcal M_k(\Bbb C)$. By
Lemma 5.3, we have
$$
 k\left (\frac 2 {m^3}\right)^2\ge  Tr((A-U BU^*)^2)\ge \sum_{i\ \in \hat T_j} |\lambda_i-j|^2\ge \left
  ( \frac 1m  \right)^2 \ card(\hat T_j),
$$  where $card(\hat T_j)$ is the cardinality of the set $\hat T_j$. Thus
$$
card (\hat T_j) \le \frac {4k}{m^4}, \qquad \text { for  } 1\le j\le
m.
$$
Let $s_j= card (T_j)$ for $1\le j\le m$, whence
  $$ \frac km+\frac k r\ge k_j\ge s_j=k_j-card(\tilde T_j)\ge k_j- \frac {4k}{m^4}\ge \frac k m -\frac k r -\frac {4k}{m^4}, \qquad \forall \ \ 1\le j\le m. $$

Let
$$T_{m+1}= \{1,2,\ldots,k\}\setminus (\cup_{j=1}^nT_j)$$
 and $s_{m+1}$ be the cardinality of the set $T_{m+1}$. Thus
 $$s_{m+1}=k-s_1-\cdots-s_m=\sum_{j=1}^m card({\hat T_j})\le \sum_{j=1}^m \frac {4k}{m^4}=\frac {4k}{m^3}.$$

It is not hard to see that $T_1,\ldots, T_{m+1}$ is a partition of
the set $\{1,2,\ldots,k\}$. Moreover, if $1\le j_1\ne j_2\le m$ then
for any
$$
   i_1\in T_{j_1} , \qquad \text { and } \qquad i_2\in T_{j_2}
$$ we have
$$
|\lambda_{i_1}- \lambda_{i_2} |\ge  |j_2-j_1|- |\lambda_{i_2}-j_2|-|
\lambda_{i_1}-j_1 |\ge 1-\frac {2 }{m }\ge \frac 1 2.
$$
Applying Proposition 3.1 for such $T_1,\ldots, T_m,T_{m+1}$,
$\theta=1/2$ and $\omega=\delta/2$, we have
$$\begin{aligned}
\nu_2(\Omega(A),\omega) &\ge  (8C_1\omega)^{-k^2}\cdot \left
(\frac{2C }{ \omega}\right )^{-2
s_1^2-\cdots-2s_m^2-2s_{m+1}^2-4(s_1+\cdots+s_m)s_{m+1}
 }\\
   & \ge  (8C_1\omega)^{-k^2}\cdot \left
(\frac{2C }{ \omega}\right )^{-2(k_1^2+\cdots+k_m^2+(\frac
{4k}{m^3})^2+2k\cdot \frac {4k}{m^3}
)}\\
&\ge (8C_1\omega)^{-k^2}\cdot \left (\frac{2C }{ \omega}\right
)^{-2((\frac k m+\frac k r)^2+\cdots+(\frac k m+\frac k r)^2+\frac
{16k^2}{m^6}+ \frac {8k^2}{m^3})}
\\
&\ge (8C_1\omega)^{-k^2}\cdot \left (\frac{2C }{ \omega}\right
)^{\frac{-56 k^2}{m}},
\end{aligned}
$$ for some   constants $C, C_1>1$  independent of $k,\omega$.

\end{proof}

\begin{lemma}
 Let $\mathcal A$ be a finitely generated, infinite dimensional, simple
unital C$^*$ algebra  with a unique tracial state $\tau$. Assume
that $A$ has approximation property. Assume that $x_1,\ldots, x_n$
is a family of self-adjoint generators of $\mathcal A$. Let $H$,
$\mathcal M$ be defined as above. Let $m$ be a positive integer. Let
$y_m$ and $P_{r_m}(x_1,\ldots, x_n)$ be chosen as above. Let
$R>\max\{\| P_{r_m}(x_1,\ldots, x_n) \|, \|x_1\|, \ldots,
\|x_n\|\}$. When $r$ is large enough and $\epsilon$ is small enough,
for any $\omega>0$, we have
$$
   \nu_2(\Gamma^{(top)}_R(P_{r_m}(x_1,\ldots, x_n):x_1,\ldots, x_n;
   k, \epsilon,  P_1,\ldots, P_r),\omega)\ge (8C_1\omega)^{-k^2}\cdot \left (\frac{2C }{ \omega}\right
)^{\frac{-56 k^2}{m}}
$$
\end{lemma}
\begin{proof}
By Lemma 5.2, when $r$ is large enough and $\epsilon$ is small
enough,  the following hold: $\forall \ k\ge 1 $, if
$$
(B,A_1,\ldots,A_n)\in \Gamma^{(top)}_R(P_{r_m}(x_1,\ldots, x_n),
x_1,\ldots,x_n; k, \epsilon, P_1,\ldots, P_r),
$$ then there are some $1\le k_1,\ldots,k_m\le k$ with $\frac 1 m -\frac 1 r\le \frac {k_j} k \le \frac 1 m +\frac 1
  r  \text {  for each } \ 1\le j\le m  $ and $k_1+\cdots+k_m=k$, and a unitary matrix
$U$ in $\mathcal U(k)$ satisfying
 $$
    \|B- U \left (
          \begin{aligned}
              \begin{aligned}
    1 \cdot I_{k_1} \quad  & \quad 0  \quad & \quad  \cdots  \quad  &  \quad   0 \\
    0 \quad  & \quad   2 \cdot I_{k_2} \quad  & \quad  \cdots  \quad  &  \quad 0 \\
      \cdots \quad& \quad \cdots & \quad  \ddots\quad & \quad\cdots\\
    0  \quad &  \quad 0 \quad  &  \quad \cdots  \quad  &
    \quad   m\cdot I_{k_{m}}
\end{aligned}
          \end{aligned}
      \right ) U^*\|_2\le \frac 2{m^3}.
  $$
 Combining with Lemma 5.4, we know that if
  $$
   B\in \Gamma^{(top)}_R(P_{r_m}(x_1,\ldots,
x_n):x_1,\ldots,x_n; k, \epsilon, P_1,\ldots, P_r)
  $$
then, for any $ \omega>0$,
$$ \nu_2(\Omega(B),\omega)\ge (8C_1\omega)^{-k^2}\cdot \left (\frac{2C }{ \omega}\right
)^{\frac{-56 k^2}{m}},
$$ where
$$
\Omega(B) =\{W^*BW \ | \ W \in \mathcal U(k)\}.
$$
Note that $\Omega(B)\subset \Gamma^{(top)}_R(P_{r_m}(x_1,\ldots,
x_n):x_1,\ldots,x_n; k,r,\epsilon)$. It follows that, for any
$\omega>0$,
$$
   \nu_2(\Gamma^{(top)}_R(P_{r_m}(x_1,\ldots, x_n):x_1,\ldots, x_n;
   k,r,\epsilon),\omega) \ge (8C_1\omega)^{-k^2}\cdot \left (\frac{2C }{ \omega}\right
)^{\frac{-56 k^2}{m}}
$$

\end{proof}

Now we have the following result.
\begin{theorem}
 Let $\mathcal A$ be a finitely generated, infinite dimensional, simple
unital C$^*$ algebra  with a unique tracial state $\tau$. Assume
that $x_1,\ldots, x_n$ is a family of self-adjoint generators of
$\mathcal A$. If $\mathcal A$ has approximation property, then
$\delta_{top}(x_1,\ldots,x_n)\ge 1$.
\end{theorem}
\begin{proof}
Let $H$ be the Hilbert space $L^2(\mathcal A,\tau)$. Without loss of
generality, we might assume that $\mathcal A$ is faithfully
represented on the Hilbert space $H$. Let $\mathcal M$ be the von
Neumann algebra generated by $\mathcal A$ on $H$. It is not hard to
see that $\mathcal M$ is a diffuse von Neumann algebra with a
tracial state $\tau$.  For each positive integer $m$, there is a
family of mutually orthogonal projections $p_1,\ldots, p_m$ in
$\mathcal M$ such that $\tau(p_j)=1/m$ for $1\le j\le m$. Let
$$y_m=1\cdot  p_1+ 2\cdot p_2+\cdots+ m\cdot p_m=\sum_{j=1}^m j\cdot
p_j.$$  Let $\{P_r(x_1,\ldots,x_n)\}_{r=1}^\infty$ be defined as in
section 2.3. Thus $\{P_r(x_1,\ldots,x_n)\}_{r=1}^\infty$ is dense in
$\mathcal M$ with respect to the strong operator topology. Hence,
for each $m\ge 1$, there is some self-adjoint element
$P_{r_m}(x_1,\ldots, x_n)$ in $\mathcal A$ such that
$$
\|y_m-P_{r_m}(x_1,\ldots, x_n)\|_2\le \frac 1 {m^3}.
$$

By Lemma 5.5, for any $\omega>0$, when  $r$ is large enough and
$\epsilon$ is small enough,  we have for some   constants $C_1, C>1$
independent of $k, \omega$
$$
   \nu_2(\Gamma^{(top)}_R(P_{r_m}(x_1,\ldots, x_n):x_1,\ldots, x_n;
   k, \epsilon, P_1,\ldots, P_r),\omega) \ge (8C_1\omega)^{-k^2}\cdot \left (\frac{2C }{ \omega}\right
)^{\frac{-56k^2}{m}}
$$
Therefore,
$$
\tilde \delta_{top}(P_{r_m}(x_1,\ldots, x_n):x_1,\ldots, x_n)\ge
1-\frac {56} m.
$$
By Proposition 5.1, we get
$$
 \delta_{top}(P_{r_m}(x_1,\ldots, x_n):x_1,\ldots, x_n)\ge 1-\frac
{56} m.
$$By Lemma 5.1,
$$
 \delta_{top}( x_1,\ldots, x_n)\ge 1-\frac
{56} m.
$$ Since $m$ is an arbitrary positive integer, we obtain
$$
\delta_{top}( x_1,\ldots, x_n)\ge 1.
$$
\end{proof}

%\begin{remark}
%Now we will be able to compute the low-bound of topological free
%entropy dimension for some unital C$^*$ algebras. For example,
%$\delta_{top}(x_1,\ldots,x_n)\ge 1$ if $x_1,\ldots, x_n$ is a family
%of self-adjoint operators that generates an irrational rotation
%algebra $\mathcal A$.
%\end{remark}
%\begin{remark}
%Combining Remark 5.1 and 5.2 we have the following:
%$\delta_{top}(x_1,\ldots,x_n)= 1$ if $x_1,\ldots, x_n$ is a family
%of self-adjoint operators that generates an irrational rotation
%algebra $\mathcal A$.
%\end{remark}

\subsection{Values of topological free entropy dimensions in some
unital C$^*$ algebras}

In this subsection, we are going to compute the values of
topological free entropy dimensions in some unital C$^*$ algebras by
using the results from preceding subsection.

\begin{theorem}
Let $\mathcal A_\theta$ be an irrational rotation C$^*$ algebra.
Then
$$\delta_{top}(x_1,\ldots,x_n)= 1$$ where $x_1,\ldots, x_n$ is a family
of self-adjoint operators that generates   $\mathcal A_\theta$.

\end{theorem}

\begin{proof}
Note that $\mathcal A_\theta$ is an infinite dimensional, unital
simple C$^*$ algebra  with a unique tracial state $\tau$. By
\cite{V3} or \cite {HaSh} and Theorem 5.1, we know that
$$
\delta_{top}(x_1,\ldots,x_n)\le 1.
$$
It follows from \cite {PiVo} that $\mathcal A_\theta$ has
approximation property. Therefore
$$
\delta_{top}(x_1,\ldots,x_n)\ge 1.
$$
Hence $$ \delta_{top}(x_1,\ldots,x_n)= 1.
$$
\end{proof}

\begin{theorem}
Let $\mathcal A$ be a UHF algebra (uniformly hyperfinite C$^*$
algebra). Then
$$\delta_{top}(x_1,\ldots,x_n)= 1$$ where $x_1,\ldots, x_n$ is a family
of self-adjoint operators that generates   $\mathcal A$.

\end{theorem}

\begin{proof}By \cite{Olsen}, we know that $\mathcal A$ is generated
by two self-adjoint elements. It is not hard to see that $\mathcal
A$ is an infinite dimensional, unital simple C$^*$ algebra  with a
unique tracial state $\tau$. By \cite{V3} or \cite {HaSh} and
Theorem 5.1, we know that
$$
\delta_{top}(x_1,\ldots,x_n)\le 1.
$$
It is easy to check  that $\mathcal A$ has approximation property.
Therefore
$$
\delta_{top}(x_1,\ldots,x_n)\ge 1.
$$
Hence $$ \delta_{top}(x_1,\ldots,x_n)= 1.
$$
\end{proof}

Recall that for any sequence $(\mathcal A_m)_{m=1}^\infty$ of C$^*$
algebras,we can introduce two C$^*$ algebras
$$
  \begin{aligned}
     \prod_m \mathcal A_m &= \{(a_m)_{m=1}^\infty\ | \ a_m \in \mathcal A_m , \
     \sup_{m\in \Bbb N} \|a_m\|<\infty\}\\
      \sum_m \mathcal A_m &= \{(a_m)_{m=1}^\infty\ | \ a_m \in \mathcal A_m , \
     \lim_{m\rightarrow \Bbb N} \|a_m\|=0\}\\
  \end{aligned}
$$
The norm in the quotient C$^*$ algebra $\prod_m \mathcal A_m/\sum_m
\mathcal A_m $ is given by
$$
   \|\rho ((a_m)_{m=1}^\infty)\|=\limsup_{m\rightarrow \infty}
   \|x_m\|,
$$ where $\rho$ is the quotient map from $\prod_m \mathcal A_m$ onto
$\sum_m \mathcal A_m $.

If $\mathcal A$ is an exact C$^*$ algebra, then the sequence
$$
0\rightarrow \mathcal A\otimes_{min} \sum_m \mathcal \mathcal
M_m(\Bbb C) \rightarrow \mathcal A\otimes_{min} \prod_m \mathcal
\mathcal M_m(\Bbb C)\rightarrow \mathcal A\otimes_{min} (\prod_m
\mathcal \mathcal M_m(\Bbb C)/\sum_m \mathcal \mathcal M_m(\Bbb
C))\rightarrow 0
$$ is exact. Therefore, we have the following natural
identification
$$
\mathcal A \otimes_{min} (\prod_m \mathcal \mathcal M_m(\Bbb
C)/\sum_m \mathcal \mathcal M_m(\Bbb C)) =  (\mathcal A
\otimes_{min} \prod_m \mathcal \mathcal M_m(\Bbb C))/ (\mathcal A
\otimes_{min} \sum_m \mathcal \mathcal M_m(\Bbb C)).
$$
On the other hand, we have the following natural embedding
$$
\mathcal A \otimes_{min} \prod_m \mathcal \mathcal M_m(\Bbb C)
\subseteq \prod_m \mathcal M_m(\mathcal A)
$$ and the identification
$$
  \mathcal A \otimes_{min} \sum_m \mathcal \mathcal M_m(\Bbb C)
= \sum_m \mathcal M_m(\mathcal A)
$$
Thus we have for any exact C$^*$ algebra $\mathcal A$ a natural
embedding
$$
\psi: \quad \mathcal A\otimes_{min} (\prod_m \mathcal \mathcal
M_m(\Bbb C)/\sum_m \mathcal \mathcal M_m(\Bbb C) \subseteq \prod_m
\mathcal \mathcal M_m(\mathcal A)/\sum_m \mathcal \mathcal
M_m(\mathcal A).
$$

\begin{lemma}
Suppose that $\mathcal A$ and $\mathcal B$ are unital C$^*$ algebras
and $\rho$ is an unital  embedding  $$\rho: \quad \mathcal
A\rightarrow \prod_m \mathcal \mathcal M_m(\mathcal B)/\sum_m
\mathcal \mathcal M_m(\mathcal B).
$$ Suppose that $x_1,\ldots,x_n$ is a family of elements in
$\mathcal A$. Suppose $r$ is a positive integer and
$\{P_j(x_1,\ldots,x_n)\}_{j=1}^r$ is a
 family of noncommutative polynomials of $x_1,\ldots,x_n$. Then
 there are some $k\in \Bbb N$ and $a_1^{(k)}, \ldots, a_n^{(k)}$ in
 $\mathcal M_k(\mathcal B)$ so that
 $$
|\|P_j(a_1^{(k)}, \ldots, a_n^{(k)})\|-\|P_j(x_1,\ldots,x_n)\||\le
\frac 1 r, \quad \forall \ 1\le j\le r.
 $$
\end{lemma}
\begin{proof}
We might assume that   $$\rho(x_i)= [(x_i^{(m)})_m]  \ \in \ \prod_m
\mathcal \mathcal M_m(\mathcal B)/\sum_m \mathcal \mathcal
M_m(\mathcal B), \quad \forall \ 1\le i\le n. $$ By the definition
of $\prod_m \mathcal \mathcal M_m(\mathcal B)/\sum_m \mathcal
\mathcal M_m(\mathcal B)$, there are some positive integers $m_1\le
m_2$  such that
$$
|(\sup_{m_1\le l\le m_2}\|P_j(x_1^{(l)}, \ldots,
x_n^{(l)})\|)-\|P_j(x_1,\ldots,x_n)\||\le \frac 1 r, \quad \forall \
1\le j\le r.
 $$
 Let $k=\sum_{j=m_1}^{m_2} j$ and
 $$
 a_i^{(k)} = \oplus_{l=m_1}^{m_2}x_i^{(l)} \in \mathcal M_k(\mathcal
 B), \qquad \forall \ 1\le i\le n.
 $$
 Then, it is not hard to check that
$$
|\|P_j(a_1^{(k)}, \ldots, a_n^{(k)})\|-\|P_j(x_1,\ldots,x_n)\||\le
\frac 1 r, \quad \forall \ 1\le j\le r.
 $$

\end{proof}

\begin{theorem}
Let $p\ge 2$ be a positive integer and $F_p$ be the free group on
$p$ generators. Let $C_{red}^*(F_p) \otimes_{min}C_{red}^*(F_p) $ be
a minimal tensor product of two reduced C$^*$ algebras of free
groups $F_p$. Then
$$
\delta_{top} (x_1,\ldots, x_n)=1,
$$ where $x_1,\ldots,x_n$ is any family of self-adjoint generators
of $C_{red}^*(F_p) \otimes_{min}C_{red}^*(F_p)$.

\end{theorem}
\begin{proof}
Note that $C_{red}^*(F_p) \otimes_{min}C_{red}^*(F_p)$ is an
infinite dimensional, unital simple C$^*$ algebra with a unique
tracial state. By the result from \cite {Ge2} or \cite {HaSh} and
Theorem 5.1, Theorem 5.2, to show $  \delta_{top} (x_1,\ldots,
x_n)=1, $  we need only to show that $C_{red}^*(F_p)
\otimes_{min}C_{red}^*(F_p)$ has approximation property. Therefore,
it suffices to show the following: Let $R >\max\{\|x_1\|,\ldots,
\|x_n\|\}$. For any $r\ge 1$, there is some $k\in \Bbb N$ so that
$$
\Gamma_R^{(top)}(x_1,\ldots,x_n;k,\frac 1 r, P_1,\ldots, P_r) \ne
\emptyset.
$$

By the result from \cite{Haag}, we know there is a unital
embedding %$\phi_1$ from $C_{red}^*(F_p)$ into $\prod_m
%\mathcal \mathcal M_m(\Bbb C)/\sum_m \mathcal \mathcal M_m(\Bbb C)$:
$$
\phi_1 : \ \quad C_{red}^*(F_p)\rightarrow \prod_m   \mathcal M_{
m}(\Bbb C)/\sum_m \mathcal   M_{ m}(\Bbb C),
$$ which induce a unital embedding
$$
\phi_2 : \ \quad
C_{red}^*(F_p)\otimes_{min}C_{red}^*(F_p)\rightarrow
C_{red}^*(F_p)\otimes_{min}(\prod_m  \mathcal M_{ m}(\Bbb C)/\sum_m
\mathcal M_{ m}(\Bbb C))
$$
Note that $C_{red}^*(F_p)$ is an exact C$^*$ algebra. From the
explanation preceding the theorem it follows that there is a unital
embedding
$$
\phi_3 : \ \quad
C_{red}^*(F_p)\otimes_{min}C_{red}^*(F_p)\rightarrow \prod_m
  \mathcal M_{m}(C_{red}^*(F_p))/\sum_m   \mathcal
M_{m}(C_{red}^*(F_p)).
$$
By Lemma 5.6, for a family of elements $x_1,\ldots, x_n$ in
$C_{red}^*(F_p)\otimes_{min}C_{red}^*(F_p)$ and $r\ge 1$, there are
some $m\in \Bbb N$ and some $a_1^{(m)},\ldots, a_n^{(m)}$ in
$\mathcal M_{m}(C_{red}^*(F_p))$ so that
$\max\{\|a_1\|,\ldots,\|a_n\|\}< R$ and
$$
| \|P_j(a_1^{(m)},\ldots, a_n^{(m)})\|- \|P_j(x_1,\ldots,x_n)\||\le
\frac 1 {2r}, \qquad \forall \ 0 \le j \le r.
$$

On the other hand, by the existence of embedding
$$
\phi_1 : \ \quad C_{red}^*(F_p)\rightarrow \prod_{m'}  \mathcal M_{
{m'}}(\Bbb C)/\sum_{m'}   \mathcal M_{{m'}}(\Bbb C),
$$ it follows that there is a unital embedding
$$
\phi_4 : \ \mathcal M_{m}(C_{red}^*(F_p))  =\mathcal M_{m}(\Bbb
C)\otimes_{min}C_{red}^*(F_p)\rightarrow \mathcal M_{m}(\Bbb
C)\otimes_{min}(\prod_{m'}   \mathcal M_{ {m'}}(\Bbb
C)/\sum_{m'}   \mathcal M_{ {m'}}(\Bbb C)) %\qquad = \prod_{m'}
%\mathcal \mathcal M_{m'm}(\Bbb C)/\sum_{m'} \mathcal \mathcal
%M_{m'm}(\Bbb C).
$$
But
$$
\mathcal M_{m}(\Bbb C)\otimes_{min}(\prod_{m'}   \mathcal M_{
{m'}}(\Bbb C)/\sum_{m'}  \mathcal M_{ {m'}}(\Bbb C)) = \prod_{m'}
\mathcal M_{ {m'}m}(\Bbb C)/\sum_{m'}
  \mathcal M_{ {m'}m}(\Bbb C).
$$
Hence for such $a_1^{(m)},\ldots, a_n^{(m)}$ in $\mathcal
M_{m}(C_{red}^*(F_p))$ and $r\ge 1$, by Lemma 5.6, there are some
$k\in \Bbb N$ and $A_1,\ldots, A_n$ in $\mathcal M_k(\Bbb C)$ so
that $\max\{\|A_1\|,\ldots,\|A_n\|\}<R$ and
$$
| \|P_j(a_1^{(m)},\ldots, a_n^{(m)})\|- \|P_j(A_1,\ldots,A_n)\||\le
\frac 1 {2r}, \qquad \forall \ 0 \le j \le r.
$$

Altogether,  we have
$$
|  \|P_j(x_1,\ldots,x_n)\|- \|P_j(A_1,\ldots,A_n)\||\le \frac 1 {
r}, \qquad \forall \ 0 \le j \le r,
$$ which  implies that $C_{red}^*(F_p)\otimes_{min}C_{red}^*(F_p)$  has approximation property.

Hence $$ \delta_{top} (x_1,\ldots, x_n)=1,
$$ for any family of self-adjoint elements $x_1,\ldots,x_n$  that
generates  $C_{red}^*(F_p) \otimes_{min}C_{red}^*(F_p)$.
\end{proof}

\begin{theorem}
Suppose that $\mathcal{K}$ be the C$^*$ algebra consisting of all
compact operators on a separable Hilbert space $H$. Suppose
$\mathcal A= \Bbb C \bigoplus\mathcal K   $ is the unitization of
$\mathcal K$. If $x_1,\ldots, x_n$ is a family of self-adjoint
elements  that generate $\mathcal A$ as a C$^*$ algebra, then
$$
\delta_{top}(x_1,\ldots,x_n)=0.
$$

\end{theorem}

\begin{proof}
By \cite  {Olsen}, we know that unital C$^*$ algebra $\mathcal A$ is
generated by two self-adjoint elements in $\mathcal A$. Note that
$\mathcal A$ has a unique trace $\tau$, which is defined by
$$
\tau((\lambda, x))= \lambda, \qquad \forall \ (\lambda,x) \in
\mathcal A.
$$
By Theorem 5.1, it is not hard to see that
$$
\delta_{top}(x_1,\ldots,x_n)=0,
$$ where $x_1,\ldots, x_n$ is a family of self-adjoint generators of $\mathcal
A$.

\end{proof}

\section{Topological free orbit  dimension of C$^*$ algebras}
Assume that $\mathcal A$ is a unital C$^*$-algebra. Let $x_1,\ldots,
x_n,$ $y_1,\ldots, y_m$ be self-adjoint elements in $\mathcal A$.
Let $\Bbb C\langle X_1,\ldots, X_n, Y_1,\ldots,Y_m\rangle $ be the
noncommutative polynomials in the indeterminates $X_1,\ldots, X_n,$
$ Y_1,\ldots,Y_m$. Let $\{P_r\}_{r=1}^\infty$ be the collection of
all noncommutative polynomials in $\Bbb C\langle X_1,\ldots, X_n,$ $
Y_1,\ldots,Y_m\rangle $ with rational coefficients.

\subsection{Unitary orbits of balls in $\mathcal M_k(\Bbb C)^n$}

We let $\mathcal{M}_{k}(\mathbb{C})$ be the $k\times k$ full matrix
algebra with entries in $\mathbb{C}$, and
  $\mathcal{U}(k)$ be the group
of all unitary matrices in $\mathcal{M}_{k}(\mathbb{C})$. Let
$\mathcal{M}_{k}(\mathbb{C})^{n}$ denote the direct sum of $n$
copies of $\mathcal{M}_{k}(\mathbb{C})$.  Let $\mathcal
M_k^{s.a}(\Bbb C)$ be the subalgebra of $\mathcal M_k(\Bbb C)$
consisting of all self-adjoint matrices of $\mathcal M_k(\Bbb C)$.
Let $(\mathcal M_k^{s.a}(\Bbb C))^n$ be the direct sum of $n$ copies
of $\mathcal M_k^{s.a}(\Bbb C)$.

For every $\omega>0$, we define the $\omega$-orbit-$\|\cdot\|$-ball $\mathcal{U}%
(B_{1},\ldots,B_{n};\omega)$ centered at $(B_{1},\ldots,B_{n})$ in $\mathcal{M}%
_{k}(\mathbb{C})^{n}$ to be the subset of
$\mathcal{M}_{k}(\mathbb{C})^{n}$ consisting of all
$(A_{1},\ldots,A_{n})$ in $\mathcal{M}_{k}(\mathbb{C})^{n}$ such
that there exists some unitary matrix $W$ in $\mathcal{U}(k)$
satisfying
\[
\Vert(A_{1},\ldots,A_{n})-(WB_{1}W^{\ast},\ldots,WB_{n}W^{\ast})\Vert
 <\omega.
\]

\subsection{Norm-microstate space}

 For all integers $r, k\ge 1$, real
numbers $R, \epsilon>0$ and noncommutative polynomials $P_1,\ldots,
P_r$, we let
$$
\Gamma^{(top)}_R(x_1,\ldots, x_n: y_1,\ldots, y_m; k,\epsilon,
P_1,\ldots, P_r)
$$ be as defined as in section 2.4.
\subsection{Topological free orbit  dimension}
\begin{definition}
For $\omega>0$, we define the   covering number
$$o_\infty(\Gamma^{(top)}
_{R}(x_{1},\ldots,x_{n}:y_1,\ldots,y_p;k,\epsilon,P_1,\ldots,P_r),\omega)$$
to be the minimal number of $\omega$-orbit--$\|\cdot\|$-balls that
cover $\Gamma^{(top)}
_{R}(x_{1},\ldots,x_{n}:y_1,\ldots,y_p;k,\epsilon, P_1,\ldots,P_r)$
with the centers of these $\omega$-orbit-$\|\cdot\|$-balls in
$\Gamma^{(top)} _{R}(x_{1},\ldots,x_{n}:y_1,\ldots,y_p;k,\epsilon,
P_1,\ldots,P_r)$

For each function $f: \ \Bbb N\times \Bbb N\times \Bbb R^+
\rightarrow \Bbb R $,  we define,
\[
\begin{aligned}
\frak k_f &(x_{1},\ldots, x_{n}:y_1,\ldots,y_p; \omega,R)\\ &=
\inf_{r\in \Bbb N, \epsilon>0}\ \ \limsup_{k\rightarrow \infty}
  {f( o_\infty(\Gamma^{(top)}
_{R}(x_{1},\ldots,x_{n}:y_1,\ldots,y_p;k,\epsilon,P_1,\ldots,P_r),\omega), k,\omega )} \\
\end{aligned}
\]and \[
\begin{aligned}
\frak k_f  &(x_{1},\ldots, x_{n}:y_1,\ldots,y_p; \omega)  =
\sup_{R>0} \frak k_f (x_{1},\ldots, x_{n}:y_1,\ldots,y_p; \omega,R)\\
\frak k_f &(x_{1},\ldots, x_{n}:y_1,\ldots,y_p )  =
\limsup_{\omega\rightarrow 0^+}\frak k_f (x_{1},\ldots,
x_{n}:y_1,\ldots,y_p; \omega ),
\end{aligned}
\]
where  $\mathfrak{k}_{f}(x_{1},,\ldots,x_{n}:y_1,\ldots,y_p)$   is
called the {\em topological $f(\cdot)$-free-orbit-dimension} of \
$x_{1},\ldots,x_{n}$ in the presence of $y_1,\ldots,y_p$.

\end{definition}

\subsection{Topological free entropy dimension and topological free orbit
dimension} The following result follows directly from the
definitions of topological free entropy dimension and topological
free orbit dimension of $n$-tuple of self-adjoint elements in a
C$^*$ algebra.
\begin{theorem}
Suppose that $\mathcal A$ is a unital C$^*$ algebra and $x_1,\ldots,
x_n$ is a family of self-adjoint elements of $\mathcal A$.
 Let $f: \Bbb N\times\Bbb N\times \Bbb
R^+\rightarrow \Bbb R $ be defined by $$f(s,k,\omega)=  \frac {\log
s}{ - k^2\log\omega}$$ for $s,k\in \Bbb N$, $\omega>0$. Then
$$
\delta_{top} (x_1,\ldots,x_n)\le \frak k_f(x_1,\ldots, x_n)+1.
$$

\end{theorem}

%\subsection{Topological upper orbit  dimension}
%\begin{definition}

%For each function $f: \ \Bbb N \times \Bbb N\rightarrow \Bbb R $, we
%define,
%\[
%\begin{aligned}
%\frak K_f &(x_{1},\ldots, x_{n}:y_1,\ldots,y_p; \omega,R)\\ &=
%\inf_{r\in \Bbb N, \epsilon>0}\ \ \limsup_{k\rightarrow \infty}f(
%o_\infty(\Gamma^{(top)}
%_{R}(x_{1},\ldots,x_{n}:y_1,\ldots,y_p;k,\epsilon,P_1,\ldots,P_r),\omega), k)\\
%\end{aligned}
%\]and \[
%\begin{aligned}
%\frak K_f  &(x_{1},\ldots, x_{n}:y_1,\ldots,y_p; \omega)  =
%\sup_{R>0} \frak K (x_{1},\ldots, x_{n}:y_1,\ldots,y_p; \omega,R)\\
%\frak K_f &(x_{1},\ldots, x_{n}:y_1,\ldots,y_p )  =
%\sup_{\omega>0}\frak K (x_{1},\ldots, x_{n}:y_1,\ldots,y_p; \omega
%),
%\end{aligned}
%\]
%where  $\mathfrak{K}_{f}(x_{1},,\ldots,x_{n}:y_1,\ldots,y_p)$   is
%called the {\em topological $f(\cdot)$-upper-orbit-dimension} of \
%$x_{1},\ldots,x_{n}$ in the presence of $y_1,\ldots,y_p$.

%\end{definition}

\section{Topological free orbit dimension of one variable}
We recall the packing number of a set in a metric space as follows.
\begin{definition}
Suppose that $X$ is a metric space with a metric distance $d$. (i)
The packing number of a set $K$ by $\omega$-balls in $X$, denoted by
$P(K,\omega)$, is the maximal cardinality of the subsets $F $ in $K$
satisfying for all $a,b$ in $F$ either $a=b$ or $d(a,b)\ge \omega$.
(ii) The packing dimension of the set $K$ in $X$, denoted by $d(K)$,
is defined by
$$ d(K) = \limsup_{\omega\rightarrow 0^+} \frac {\log(P(K,\omega))}{-\log\omega}.    $$
\end{definition}

\subsection{Upper-bound of the topological free orbit dimension of one
variable} Suppose that $x=x^*$ is a self-adjoint element in a unital
C$^*$ algebra $\mathcal A$ and $\sigma(x)$ is the spectrum of $x$ in
$\mathcal A$.

For any $\omega>0$, let $m=P(K,\omega)$ be the packing number of
$\sigma(x)$ in $\Bbb R$. Thus there exists a family of elements
$\lambda_1, \ldots, \lambda_m$ in $\sigma(x)$ such that (i)
$|\lambda_i-\lambda_j|\ge \omega$ for all $1\le i\ne j\le m$; and
(ii) for any $\lambda$ in $\sigma(x)$, there is some $\lambda_j$
with $1 \le j\le m$ satisfying $|\lambda-\lambda_j|\le \omega$.

\begin{lemma}
  For any given
$R>\|x\|$, when $r$ is large enough and $\epsilon$ is small enough,
we have
    $$
\limsup_{k\rightarrow\infty} \frac {\log
o_{\infty}(\Gamma^{(top)}_R(x; k,\epsilon, P_1,\ldots, P_r),3\omega
)}{\log k} \le m.$$
\end{lemma}
\begin{proof}
   By Theorem 3.1, there exist some $r_0\ge 1$ and $\epsilon_0>0$ such
   that
the following holds: when
   $r>r_0$, $\epsilon<\epsilon_0$, for any $A$ in
   $
\Gamma^{(top)}_R(x; k,\epsilon, P_1,\ldots, P_r) $, there are
positive integers $1\le k_1,\ldots, k_m\le k$ with $k_1+\cdots
+k_{m}=k$  and some unitary matrix $U$ in $\mathcal M_k(\Bbb C)$
satisfying
$$
\|U^*AU- \left (  \begin{aligned}
   \lambda_1 I_{k_1} \quad  & \quad 0  \quad & \quad  \cdots  \quad  &  \quad  0\quad\\
    0 \quad  & \quad  \lambda_2 I_{k_2} \quad  & \quad  \cdots  \quad  &  \quad 0\quad\\
   \cdots \quad &  \quad\cdots \quad  &  \quad \ddots  \quad &  \quad
    \cdots \\
    0  \quad &  \quad 0 \quad  &  \quad \cdots  \quad  &
    \quad  \lambda_{m}I_{k_{m}}
\end{aligned} \right )\| \le 2\omega,
$$ where  $I_{k_j}$ is the $k_j\times k_j$ identity matrix for $1\le
j\le m$.

Let $$ \begin{aligned} \Omega(k_1,\ldots,k_m) &=\left \{ U^* \left (
\begin{aligned}
   \lambda_1 I_{k_1} \quad  & \quad 0  \quad & \quad  \cdots  \quad  &  \quad  0\quad\\
    0 \quad  & \quad  \lambda_2 I_{k_2} \quad  & \quad  \cdots  \quad  &  \quad 0\quad\\
   \cdots \quad &  \quad\cdots \quad  &  \quad \ddots  \quad &  \quad
    \cdots \\
    0  \quad &  \quad 0 \quad  &  \quad \cdots  \quad  &
    \quad  \lambda_{m}I_{k_{m}}
\end{aligned} \right )U \ | \ \text {$U$  is in $\mathcal
U_k $} \right \}.\end{aligned}
$$

 Let $\mathcal J$ be the set consisting of all these
$(k_1,\ldots, k_m )\in \Bbb N^m$ with $k_1+\cdots +k_{m}=k$. Then
the cardinality of the set $\mathcal J$ is equal to
$$
 \frac {(k-1) !}{(m-1)! (k-m)!}.
$$

Then
$$
\Gamma^{(top)}_R(x; k,\epsilon, P_1,\ldots, P_r)$$ is contained in $
2\omega$-neighborhood of the set$$  \bigcup_{(k_1,\ldots, k_{m})\in
\mathcal J} \ \ \Omega(k_1,\ldots,k_{m}).
$$ It follows that
$$
o_{\infty}(\Gamma^{(top)}_R(x; k,\epsilon, P_1,\ldots, P_r),
3\omega) \le o_{\infty}(\bigcup_{(k_1,\ldots, k_{m})\in \mathcal J}
\ \ \Omega(k_1,\ldots,k_{m}), \omega) \le |\mathcal J|=  \frac
{(k-1) !}{(m-1)! (k-m)!}.
$$
Therefore,
$$\begin{aligned}
\limsup_{k\rightarrow\infty} \frac {\log
o_{\infty}(\Gamma^{(top)}_R(x; k,\epsilon, P_1,\ldots, P_r),
3\omega)}{ \log k} &= \limsup_{k\rightarrow\infty} \frac {\log
o_{\infty}(\Gamma^{(top)}_R(x;
k,\epsilon, P_1,\ldots, P_r), \omega )}{ \log k  }\\
&\le \limsup_{k\rightarrow\infty} \frac {\log \left ( \frac {(k-1)
!}{(m-1)! (k-m)!} \right)}{ \log k  }=  m-1. \end{aligned}
$$

\end{proof}

\subsection{Lower-bound}

Suppose that $x=x^*$ is a self-adjoint element in a unital C$^*$
algebra $\mathcal A$ and $\sigma(x)$ is the spectrum of $x$ in
$\mathcal A$.

\begin{lemma}We have
$$\begin{aligned}
\limsup_{k\rightarrow\infty} &\frac {\log
o_{\infty}(\Gamma^{(top)}_R(x;
k,\epsilon, P_1,\ldots, P_r), \frac\omega 3)}{ \log k}   \ge  m-1.\\
\end{aligned}
$$
\end{lemma}
\begin{proof}

For any $\omega>0$, let $m=P(K,\omega)$ be the packing number of
$\sigma(x)$ in $\Bbb R$. Thus there exists a family of elements
$\lambda_1, \ldots, \lambda_m$ in $\sigma(x)$ such that (i)
$|\lambda_i-\lambda_j|\ge \omega$ for all $1\le i\ne j\le m$; and
(ii) for any $\lambda$ in $\sigma(x)$, there is some $\lambda_j$
with $1 \le j\le m$ satisfying $|\lambda-\lambda_j|\le \omega$.

For any $R>\|x\|$, $r\ge 1$ and $\epsilon>0$, by functional
calculus, there are $\lambda_{m+1}, \ldots, \lambda_n$ in
$\sigma(x)$ such that  for every $1\le t_1,\ldots, t_m\le k-n$ with
$2nt_1+\ldots+2nt_m= k-n$,
 the matrix
\begin{align}A= diag(\lambda_1 I_{2nt_1}, &\lambda_2 I_{2nt_2},
\ldots, \lambda_m I_{2nt_m},\lambda_1,\ldots,\lambda_{m},\ldots,
\lambda_n)\notag\\&\qquad \qquad \qquad \qquad \qquad \qquad \qquad
\quad \text {is in } \Gamma^{(top)}_R(x;k,\epsilon, P_1,\ldots,
P_r),  \end{align} where we assume that $2n | (k-n)$.

Let $\mathcal J$ be the set consisting of all these $(t_1,\ldots,
t_m)\in \Bbb N^m$ with  $2nt_1+\ldots+ 2n t_{m}= k-n$. Then the
cardinality of the set $\mathcal J$ is equal to
$$
  \frac {\left ( \frac {k-n} {2n}-1\right
)!}{\left ( \frac {k-n} {2n}-m  \right )! (m-1) !}.
$$

By Weyl's theorem in \cite{Weyl} on the distance of unitary orbits
of two self-adjoint matrices, for any two distinct elements
$$
(s_1,\ldots, s_m ) \qquad \text { and } \qquad (t_1,\ldots, t_m )
$$ in $\mathcal J$ and any $W$ in $\mathcal U(k)$, we have
$$
  \|A_1-WA_2W^*\| \ge \omega,
$$
where
$$
  \begin{aligned}
     A_1& =
diag(\lambda_1 I_{2nt_1}, \lambda_2 I_{2nt_2}, \ldots, \lambda_m
I_{2nt_m},  \lambda_1,\ldots,\lambda_{m},\ldots, \lambda_n)\\
A_2 &= diag(\lambda_1 I_{2ns_1}, \lambda_2 I_{2ns_2}, \ldots,
\lambda_{m} I_{2ns_{m}} ,\lambda_1,\ldots,\lambda_{m},\ldots,
\lambda_n)
  \end{aligned}
$$ are two  diagonal self-adjoint matrices in $\mathcal M_k(\Bbb
C)$. Combining with  (7.2.1), we have
$$
o_{\infty}(\Gamma^{(top)}_R(x; k,\epsilon, P_1,\ldots, P_r),
\frac\omega 3) \ge |\mathcal J|\ge \frac {\left ( \frac {k-n}
{2n}-1\right )!}{\left ( \frac {k-n} {2n}-m  \right )! (m-1) !}.
$$
Hence
$$
\limsup_{k\rightarrow\infty} \frac {\log
o_{\infty}(\Gamma^{(top)}_R(x; k,\epsilon, P_1,\ldots, P_r),
\frac\omega 3)}{ \log k}   \ge \limsup_{k\rightarrow\infty} \frac
{\log \frac {\left ( \frac {k-n} {2n}-1\right )!}{\left ( \frac
{k-n} {2n}-m  \right )! (m-1) !}}{ \log k }=  m-1.
$$

\end{proof}
\subsection{Topological free orbit dimension of one self-adjoint element}

\begin{theorem}
Suppose that $x=x^*$ is a self-adjoint element in a unital C$^*$
algebra $\mathcal A$ and $\sigma(x)$ is the spectrum of $x$ in
$\mathcal A$. Let $d(\sigma(x))$ be the packing dimension of the set
$\sigma(x) $ in $\Bbb R$. Let $f: \Bbb N\times\Bbb N\times \Bbb
R^+\rightarrow \Bbb R $ be defined by $$f(s,k,\omega)= \frac
{\log\left (\frac {\log s}{\log k}\right )}{-\log\omega}$$ for
$s,k\in \Bbb N$, $\omega>0$. Then
$$
\frak k_f(x) = d(\sigma(x)).
$$
\end{theorem}
\begin{proof}
 The result follows directly from Lemma 7.1, Lemma 7.2 and Definition 7.1.
\end{proof}

\begin{theorem}
Suppose that $x=x^*$ is a self-adjoint element in a unital C$^*$
algebra $\mathcal A$.  Let $f: \Bbb N\times\Bbb N\times \Bbb
R^+\rightarrow \Bbb R $ be defined by $$f(s,k,\omega)=  \frac {\log
s}{ - k^2\log\omega}$$ for $s,k\in \Bbb N$, $\omega>0$. Then
$$
\frak k_f(x) = 0.
$$
\end{theorem}
\begin{proof}
 The result follows directly from Lemma 7.1 and Definition 7.1.
\end{proof}

\end{document}